\documentclass{amsart}
\usepackage{amsmath}
\usepackage{amssymb}

\newtheorem{theorem}{Theorem}[section]
\newtheorem{claim}{Claim}[theorem]
\newtheorem{lemma}[theorem]{Lemma}
\newtheorem{proposition}[theorem]{Proposition}
\newtheorem{observation}[theorem]{Observation}
\newtheorem{corollary}[theorem]{Corollary}

\theoremstyle{definition}
\newtheorem{definition}[theorem]{Definition}
\newtheorem{example}[theorem]{Example}

\theoremstyle{remark}
\newtheorem{remark}[theorem]{Remark}

\newtheorem{conclusion}[theorem]{Conclusion}
\newtheorem{context}[theorem]{Context}

\newcommand{\rest}{{\restriction}}
\newcommand{\dom}{{\rm dom}} 

\newcommand{\set}{{\rm set}}
\newcommand{\conc}{{}^\frown\!}
\newcommand{\val}{{\mathbf{val}}} 
\newcommand{\nor}{{\mathbf{nor}}} 
\newcommand{\dis}{{\mathbf{dis}}} 
\newcommand{\up}{{\rm up}} 
\newcommand{\dn}{{\rm dn}} 
\newcommand{\PC}{{{\rm PC}_\infty}}
\newcommand{\PCt}{{{\rm PC}_\infty^{\tt tt}}}
\newcommand{\PCJ}{{\rm PC}^J_{{\rm w}\infty}(K,\Sigma)}
\newcommand{\FC}{{{\rm FC}(K,\Sigma)}}
\newcommand{\Sigt}{{\Sigma^{\tt tt}}}
\newcommand{\Sigseq}{{\Sigma^{\rm seq}}}
\newcommand{\post}{{{\rm pos}^{\tt tt}}}
\newcommand{\upl}{\upharpoonleft}
\newcommand{\base}{{\rm base}}

\newcommand{\pos}{{\rm pos}} 
\newcommand{\uf}{{\rm uf}} 
\newcommand{\suft}{{\rm suf}_{\bar{t}}(K,\Sigma)} 
\newcommand{\sufs}{{\rm suf}_{\bar{t}}^*(K,\Sigma)} 
\newcommand{\sufp}{{\rm suf}_{\bar{t}}^\diamond(K,\Sigma)} 
\newcommand{\sufJ}{{\rm suf}_{\bar{t}}^J(K,\Sigma)} 

\newcommand{\cA}{{\mathcal A}}

\newcommand{\cB}{{\mathcal B}}

\newcommand{\cC}{{\mathcal C}}

\newcommand{\cD}{{\mathcal D}}
 
\newcommand{\bH}{{\bf H}}

\newcommand{\bfun}{{\mathbf{f}}}
\newcommand{\cF}{{\mathcal F}}

\newcommand{\ttl}{{\tt l}}

\newcommand{\cS}{{\mathcal S}}

\newcommand{\ttt}{{\tt t}}

\newcommand{\bbX}{{\mathbb X}}

\newcommand{\bbY}{{\mathbb Y}}

\newcommand{\bbZ}{{\mathbb Z}}

\newcount\skewfactor
\def\mathunderaccent#1#2 {\let\theaccent#1\skewfactor#2
\mathpalette\putaccentunder}
\def\putaccentunder#1#2{\oalign{$#1#2$\crcr\hidewidth
\vbox to.2ex{\hbox{$#1\skew\skewfactor\theaccent{}$}\vss}\hidewidth}}

\begin{document}

\title{Partition theorems from creatures and idempotent ultrafilters} 

\author{Andrzej Ros{\l}anowski}
\address{Department of Mathematics\\
 University of Nebraska at Omaha\\
 Omaha, NE 68182-0243, USA}
\email{roslanow@member.ams.org}
\urladdr{http://www.unomaha.edu/logic}

\author {Saharon Shelah}
\address{Einstein Institute of Mathematics\\
Edmond J. Safra Campus, Givat Ram\\
The Hebrew University of Jerusalem\\
Jerusalem, 91904, Israel\\
and \\
Department of Mathematics\\
Hill Center - Busch Campus \\ 
Rutgers, The State University of New Jersey \\
110 Frelinghuysen Road \\
Piscataway, NJ 08854-8019 USA}
\email{shelah@math.huji.ac.il}
\urladdr{http://shelah.logic.at}
\thanks{This research was supported by the United States-Israel
Binational Science Foundation. Publication 957.}

\date{January, 2011}

\begin{abstract}  
We show a general scheme of Ramsey-type results for partitions of countable
sets of finite functions, where ``one piece is big'' is interpreted in the
language originating in creature forcing. The heart of our proofs follows
Glazer's proof of the Hindman Theorem, so we prove the existence of
idempotent ultrafilters with respect to suitable operation. Then we deduce
partition theorems related to creature forcings.
\end{abstract}

\maketitle
\numberwithin{equation}{section}
\setcounter{section}{-1}

\section {Introduction}
A typical partition theorem asserts that if a set with some structure is
divided into some number of ``nice'' pieces, then one of the pieces is large
from the point of the structure under considerations. Sometimes, the
underlying structure is complicated and it is not immediately visible that
the arguments in hands involve a partition theorem. Such is the case with many
forcing arguments. For instance, the proofs of propernes of some forcing
notions built according to the scheme of {\em norms on possibilities\/} have
in their hearts partition theorems stating that at some situations a
homogeneous tree and/or a sequence of creatures determining a condition can
be found (see, e.g., Ros{\l}anowski and Shelah \cite{RoSh:470, RoSh:736}, 
Ros{\l}anowski, Shelah and Spinas \cite{RShS:941}, Kellner and Shelah
\cite{KrSh:872, KrSh:961}). A more explicit connection of partition theorems
with forcing arguments is given in Shelah and Zapletal \cite{ShZa:952}.

The present paper is a contribution to the Ramsey theory in the context of
finitary creature forcing. We are motivated by earlier papers and notions
concerning {\em norms on possibilities}, but we do not look at possible
forcing consequences. The common form of our results here is as follows. If
a certain family of partial finite functions is divided into finitely many
pieces, then one of the pieces contains all partial functions determined by
an object (``a pure candidate'') that can be interpreted as a forcing
condition if we look at the setting from the point of view of the creature
forcing. Sets of partial functions determined by a pure candidate might be
considered as ``large'' sets.

Our main proofs are following the celebrated Glazer's proof of the Hindman
Theorem, which reduced the problem to the existence of a relevant
ultrafilter on $\omega$ in ZFC. Those arguments were presented by Comfort in
\cite[Theorem 10.3, p.451]{Cmf77} with \cite[Lemma 10.1, p.449]{Cmf77} as a
crucial step (stated here in \ref{z4}). The arguments of the second section
of our paper really resemble Glazer's proof. In that section we deal with
the easier case of omittory--like creatures ({\em loose FFCC pairs\/} of
\ref{g1}(2)) and in the proof of the main conclusion (\ref {g11}) we use an
ultrafilter idempotent with respect to operation $\oplus$ (defined in
\ref{g6}). The third section deals with the case of {\em tight FFCC pairs}
of \ref{g1}(4). Here, we consider partitions of some sets of partial
functions all of which have domains being essentially intervals of integers
starting with some fixed $n<\omega$. While the general scheme of the
arguments follows the pattern of the second section, they are slightly more
complicated as they involve {\em sequences of ultrafilters\/} and operations
on them. As an application of this method, in \ref{CS6.3new} we give a new
proof of a partition theorem by Carlson and Simpson \cite[Theorem
6.3]{CS84}. The next section presents a variation of the third section:
under weaker assumptions on the involved FFCC pairs we get weaker, yet still
interesting partition theorem. Possible applications of this weaker version
include a special case of the partition theorem by Goldstern and Shelah
\cite{GoSh:388} (see \ref{case388}). These results motivate the fourth
section, where we develop the parallel of the very weak bigness for
candidates with ``limsup'' demand on the norms.
\medskip

Our paper is self-contained and all needed
``creature terminology'' is introduced in the first section. We also give
there several examples of creating pairs to which our results may be
applied.  
\medskip

\noindent {\bf Notation}:\quad We use standard set-theoretic notation. 

$\bullet$ An integer $n$ is the set $\{0,1,\ldots,n-1\}$ of all integers
smaller than $n$, and the set of all integers is called $\omega$. For
integers $n<m$, the interval $[n,m)$ denotes the set of all {\em integers\/}
smaller than $m$ and greater than or equal to $n$. 

$\bullet$ All sequences will be indexed by natural numbers and a sequence of
objects is typically denoted by a bar above a letter with the convention
that $\bar{x}=\langle x_i:i<y\rangle$, $y\leq \omega$. 

$\bullet$ For a set $X$ the family of all subsets of $X$ is denoted by
${\mathcal P}(X)$. The domain of a function $f$ is called $\dom(f)$. 

$\bullet$ An ideal $J$ on $\omega$ is a family of subsets of $\omega$ such
that 
\begin{enumerate}
\item[(i)] all finite subsets of $\omega$ belong to $J$ but $\omega\notin
  J$, and 
\item[(ii)] if $A\subseteq B\in J$, then $A\in J$ and if $A,B\in J$ then
  $A\cup B\in J$. 
\end{enumerate}  
For an ideal $J$, the family of all subsets of $\omega$ that do not belong
to $J$ is denoted by $J^+$, and the filter dual to $J$ is called $J^c$. 

\section{Partial creatures}
We use the context and notation of Ros{\l}anowski and Shelah
\cite{RoSh:470}, but below we recall all the required definitions and
concepts. 

Since we are interested in Ramsey-type theorems and ultrafilters on a
countable set of partial functions, we will use pure candidates rather than
forcing notions generated by creating pairs. Also, our considerations will
be restricted to creating pairs which are forgetful, smooth
(\cite[1.2.5]{RoSh:470}), monotonic (\cite[5.2.3]{RoSh:470}), strongly
finitary (\cite[1.1.3, 3.3.4]{RoSh:470}) and in some cases omittory--like
(\cite[2.1.1]{RoSh:470}). Therefore we will reformulate our definitions for
this restricted context (in particular, $\val[t]$ is a set of {\em
  partial\/} functions), thus we slightly depart from the setting of
\cite{RoSh:470}.

\begin{context}
\label{context}
In this paper $\bH$ is a fixed function defined on $\omega$ and such that
$\bH(i)$ is a finite non-empty set for each $i<\omega$. The set of all
finite non-empty functions $f$ such that $\dom(f)\subseteq \omega$ and
$f(i)\in \bH(i)$ (for all $i\in\dom(f)$) will be denoted by $\cF_\bH$.
\end{context}

\begin{definition}
\label{g1}
\begin{enumerate}
\item An FP creature\footnote{FP stands for {\em {\bf F}orgetful {\bf
        P}artial creature}} for $\bH$ is a tuple 
\[t=(\nor,\val,\dis,m_\dn,m_\up)=
(\nor[t],\val[t],\dis[t],m^t_\dn,m^t_\up)\]
such that 
\begin{itemize}
\item $\nor$ is a non-negative real number, $\dis$ is an arbitrary object
  and $m^t_\dn<m^t_\up<\omega$ and  
\item $\val$ is a non-empty finite subset of $\cF_\bH$ such that
  $\dom(f)\subseteq [m^t_\dn,m^t_\up)$ for all $f\in\val$.   
\end{itemize}
\item An FFCC pair\footnote{FFCC stands for {\em smooth {\bf F}orgetful
monotonic strongly {\bf F}initary {\bf C}reature {\bf C}reating pair}} for
$\bH$ is a pair $(K,\Sigma)$ such that   
\begin{enumerate}
\item[(a)]  $K$ is a countable family of FP creatures for $\bH$, 
\item[(b)]  for each $m<\omega$ the set $K_{\le m} := \{t\in K:m^t_\up \le
  m\}$ is finite  and the set $K_{\ge m} := \{t\in K:m^t_\dn\ge m\ \&\
  \nor[t]\geq m\}$ is infinite,
\item[(c)]  $\Sigma$ is a function with the domain $\dom(\Sigma)$ included
  in the set 
\[\{(t_0,\ldots,t_n): n<\omega,\ t_\ell \in K\mbox{ and } 
m^{t_\ell}_\up\leq m^{t_{\ell+1}}_\dn \mbox{ for }\ell < n\}\]  
and the range included in ${\mathcal P}(K)\setminus\{\emptyset\}$,
\item[(d)] if $t\in\Sigma(t_0,\ldots,t_n)$ then ($t\in K$ and)
  $m^{t_0}_\dn=m^t_\dn<m^t_\up=m^{t_n}_\up$, 
\item[(e)] $t\in \Sigma(t)$ (for each $t\in K$) and 
\item[(f)] if $t\in\Sigma(t_0,\ldots,t_n)$ and $f \in\val[t]$, then 
\[\dom(f)\subseteq \bigcup\{[m^{t_\ell}_\dn,m^{t_\ell}_\up):\ell\le n\}\]
and $f\rest [m^{t_\ell}_\dn, m^{t_\ell}_\up)\in\val[t_\ell]\cup
\{\emptyset\}$ for $\ell\leq n$, and 
\item[(g)] if $\bar{t}_0,\ldots,\bar{t}_n \in \dom(\Sigma)$ and $\bar{t}=
\bar{t}_0\conc \ldots \conc \bar{t}_n\in\dom(\Sigma)$, then
\[\bigcup\{\Sigma(s_0,\ldots,s_n):s_\ell\in \Sigma(\bar{t}_\ell)\mbox{ for }
\ell\le n\} \subseteq \Sigma(\bar{t}).\]  
\end{enumerate}
\item An FFCC pair $(K,\Sigma)$ is {\em loose\/} if 
\begin{enumerate}
\item[(c$^{\rm loose}$)]  the domain of $\Sigma$ is
\[\dom(\Sigma)=\{(t_0,\ldots,t_n): n<\omega,\ t_\ell \in K\mbox{ and }  
m^{t_\ell}_\up\leq m^{t_{\ell+1}}_\dn \mbox{ for }\ell < n\}.\]  
\end{enumerate}
\item An FFCC pair $(K,\Sigma)$ is {\em tight\/} if 
\begin{enumerate}
\item[(c$^{\rm tight}$)]  the domain of $\Sigma$ is
\[\dom(\Sigma)=\{(t_0,\ldots,t_n): n<\omega,\ t_\ell \in K\mbox{ and }  
m^{t_\ell}_\up=m^{t_{\ell+1}}_\dn \mbox{ for }\ell < n\},\]  
\item[(f$^{\rm tight}$)] if $t\in\Sigma(t_0,\ldots,t_n)$ and $f \in\val[t]$,
  then $f\rest [m^{t_\ell}_\dn,m^{t_\ell}_\up)\in \val[t_\ell]$ for all
  $\ell\leq n$, and  
\item[(h$^{\rm tight}$)] if $s_0,s_1\in K$, $m^{s_0}_\up=m^{s_1}_\dn$,
  $f_0\in\val[s_0]$, $f_1\in \val[s_1]$ and $f=f_0\cup f_1$, then there is
  $s\in\Sigma(s_0,s_1)$ such that $f\in\val[s]$. 
\end{enumerate}
\end{enumerate}
\end{definition} 

\begin{definition}
[Cf. {\cite[Definition 1.2.4]{RoSh:470}}]
\label{g3}
Let $(K,\Sigma)$ be an FFCC pair for $\bH$. 
\begin{enumerate}
\item A {\em pure candidate for $(K,\Sigma)$\/} is a sequence
  $\bar{t}=\langle t_n:n < \omega\rangle$ such that $t_n \in K$,
  $m^{t_n}_\up\le m^{t_{n+1}}_\dn$ (for $n<\omega$) and
  $\lim\limits_{n\to\infty} \nor[t_n]=\infty$.\\
A pure candidate $\bar{t}$ is {\em tight\/} if $m^{t_n}_\up
=m^{t_{n+1}}_\dn$ (for $n<\omega$).\\ 
The set of all pure candidates for $(K,\Sigma)$ is denoted by
$\PC(K,\Sigma)$ and the family of all tight pure candidates is called
$\PCt(K,\Sigma)$. 
\item For pure candidates $\bar{t},\bar{s}\in \PC(K,\Sigma)$ we write
$\bar{t} \le \bar{s}$ whenever there is a sequence $\langle u_n:n <
\omega\rangle$ of non-empty finite subsets of $\omega$ satisfying
\[\max(u_n)<\min(u_{n+1})\ \mbox{ and }\ \bar{s}_n \in
\Sigma(\bar{t}\rest u_n) \qquad \mbox{ for all }n<\omega.\]
\item For a pure candidate $\bar{t}=\langle t_i:i<\omega\rangle\in
\PC(K,\Sigma)$ we define  
\begin{enumerate}
\item[(a)] $\cS(\bar{t})=\{(t_{i_0},\ldots,t_{i_n}):i_0<\ldots <i_n <\omega$ 
for some $n <\omega\}$, and 
\item[(b)] $\Sigma'(\bar{t})=\bigcup\{\Sigma(\bar{s}): \bar{s}\in
\cS(\bar{t})\}$ and $\Sigt(\bar{t})=\bigcup\{\Sigma(t_0, \ldots,t_n):
n<\omega\}$, 
\item[(c)] $\pos(\bar{t})=\bigcup\{\val[s]:s \in \Sigma'(\bar{t})\}$ and  
$\post(\bar{t})=\bigcup\{\val[s]:s\in\Sigt(\bar{t})\}$,
\item[(d)] $\bar{t}\upl n=\langle t_{n+k}:k<\omega\rangle$.
\end{enumerate}
\end{enumerate}
(Above, if $\bar{s}\notin \dom(\Sigma)$ then we stipulate
$\Sigma(\bar{s})=\emptyset$.) 
\end{definition}

\begin{remark}
\label{rem2}
{\em Loose {\rm FFCC}} and {\em tight {\rm FFCC}} are the two cases of FFCC
pairs treated in this article. The corresponding partition theorems will be 
slightly different in the two cases, though there is a parallel. In the
loose case we will deal with $\Sigma'(\bar{t})$, $\pos(\bar{t})$ and
ultrafilters on the latter set. In the tight case we will use
$\Sigt(\bar{t})$, $\post(\bar{t})$ and {\em sequences\/} of ultrafilters on
$\post(t\upl n)$ (for $n<\omega$).

We will require two additional properties from $(K,\Sigma)$: {\em weak
  bigness\/} and {\em weak additivity\/} (see \ref{defadd},
\ref{defbig}). Because of the differences in the treatment of the two cases, 
there are slight differences in the formulation of these properties, so we
have two variants for each: {\em \ttl--variant\/} and {\em \ttt--variant\/} 
(where ``\ttl'' stands for ``loose'' and ``\ttt'' stands for ``tight'', of
course).  
  
Plainly, $\PCt(K,\Sigma)\subseteq \PC(K,\Sigma)$, $\Sigt(\bar{t}) 
\subseteq \Sigma'(\bar{t})$ and $\post(\bar{t})\subseteq
\pos(\bar{t})$. Also, if $\bar{t}\in\PCt(K,\Sigma)$, then $\bar{t}\upl n\in 
\PCt(K,\Sigma)$ for all $n<\omega$. 
\end{remark}

\begin{definition}
\label{defadd}
Let $(K,\Sigma)$ be an FFCC pair for $\bH$ and $\bar{t}= \langle
t_i:i<\omega\rangle\in\PC(K,\Sigma)$.   
\begin{enumerate}
\item We say that the pair $(K,\Sigma)$ {\em has weak \ttl--additivity for  
  the candidate $\bar{t}$\/} if for some increasing $\bfun:\omega 
\longrightarrow\omega$, for every $m<\omega$ we have:\\  
if $s_0,s_1\in \Sigma'(\bar{t})$, $\nor[s_0]\ge \bfun(m)$, $m^{s_0}_\dn\ge 
\bfun(m)$, $\nor[s_1]\ge\bfun(m^{s_0}_\up)$ and $m^{s_1}_\dn >
\bfun(m^{s_0}_\up)$, then we can find $s\in \Sigma'(\bar{t})$ such that
\[m^s_\dn\ge m,\ \nor[s] \ge m,\mbox{ and } \val[s]\subseteq \{f \cup g:f \in 
\val[s_0],\ g\in\val[s_1]\}.\]
\item The pair $(K,\Sigma)$ {\em has weak \ttt--additivity for
  the candidate $\bar{t}$\/} if for some increasing $\bfun:\omega
\longrightarrow\omega$, for every $n,m<\omega$ we have:\\  
if $s_0\in \Sigma(t_n,\ldots,t_k)$, $k\geq n$, $\nor[s_0]\ge\bfun(n+m)$,
$s_1\in\Sigma(t_{k+1},\ldots,t_\ell)$, $\nor[s_1]\ge\bfun(k+m)$ and
$\ell>k$, then we can find $s\in \Sigma(t_n,\ldots,t_\ell)$ such that
$\nor[s]\ge m$ and $\val[s]\subseteq \{f\cup g:f\in\val[s_0],\ g\in
\val[s_1]\}$.  
\item The pair $(K,\Sigma)$ {\em has \ttl--additivity\/} if for all
  $s_0,s_1\in K$ with $\nor[s_0],\nor[s_1]>1$ and $m^{s_0}_\up\leq
  m^{s_1}_\up$ there is $s\in\Sigma(s_0,s_1)$ such that 
\[\nor[s] \geq \min\{\nor[s_0],\nor[s_1]\}-1\ \mbox{ and }\ \val[s]
\subseteq \{f \cup g:f \in  \val[s_0],\ g\in\val[s_1]\}.\] 
The pair $(K,\Sigma)$ {\em has \ttt--additivity\/} if for all 
  $s_0,s_1\in K$ with $\nor[s_0],\nor[s_1]>1$ and $m^{s_0}_\up=m^{s_1}_\up$
  there is $s\in\Sigma(s_0,s_1)$ such that 
\[\nor[s] \geq\min\{\nor[s_0],\nor[s_1]\}-1.\] 
We say that $(K,\Sigma)$ {\em has \ttt--multiadditivity\/} if for all
$s_0,\ldots,s_n\in K$ with $m^{s_\ell}_\up=m^{s_{\ell+1}}_\dn$ (for
$\ell<n$) there is $s\in \Sigma(s_0,\ldots,s_n)$ such that $\nor[s]\geq
\max\{\nor[s_\ell]:\ell\leq n\}-1$.  
\end{enumerate}
\end{definition}

\begin{definition}
\label{defbig}
Let $(K,\Sigma)$ be an FFCC pair for $\bH$ and $\bar{t}= \langle
t_i:i<\omega\rangle\in\PC(K,\Sigma)$.   
\begin{enumerate}
\item We say that the pair $(K,\Sigma)$ {\em has weak \ttl--bigness for
  the candidate $\bar{t}$\/} whenever the following property is satisfied:
\begin{enumerate}
\item[$(\circledast)^{\bar{t}}_{\ttl}$] if $n_1,n_2,n_3<\omega$ and
  $\pos(\bar{t}) = \bigcup\{\cF_\ell:\ell<n_1\}$, then for some $s\in
  \Sigma'(\bar{t})$ and $\ell<n_1$ we have  
\[\nor[s] \ge n_2,\ m^s_\dn\ge n_3, \ \mbox{ and } \val[s]\subseteq \cF_\ell.\]
\end{enumerate}
\item We say that the pair $(K,\Sigma)$ {\em has weak \ttt--bigness for
  the candidate $\bar{t}$\/} whenever the following property is satisfied:
\begin{enumerate}
\item[$(\circledast)^{\bar{t}}_{\ttt}$] if $n,n_1,n_2<\omega$ and
  $\post(\bar{t} \upl n) = \bigcup\{\cF_\ell:\ell < n_1\}$, then
  for some $s\in \Sigt(\bar{t}\upl n)$ and $\ell<n_1$ we have  
\[\nor[s] \ge n_2\mbox{ and } \val[s]\subseteq \cF_\ell.\]
\end{enumerate}
\item We say that the pair $(K,\Sigma)$ {\em has bigness\/} if for every 
  creature $t\in K$ with $\nor[t] > 1$ and a partition $\val[t]= F_1 \cup
  F_2$, there are $\ell\in\{1,2\}$ and $s \in \Sigma(t)$ such that
  $\nor[s] \ge \nor[t]-1$ and $\val[s]\subseteq F_\ell$.
\end{enumerate}
\end{definition} 

\begin{definition}
\label{simple}
Let $(K,\Sigma)$ be an FFCC pair for $\bH$.
\begin{enumerate}
\item $(K,\Sigma)$ is {\em simple except omitting\/} if for every  
  $(t_0,\ldots,t_n)\in\dom(\Sigma)$ and $t\in\Sigma(t_0,\ldots,t_n)$ for
  some $\ell\leq n$ we have $\val[t]\subseteq \val[t_\ell]$. 
\item $(K,\Sigma)$ is {\em gluing on a candidate $\bar{t}= \langle
t_i:i<\omega\rangle\in\PC(K,\Sigma)$} if for every $n,m<\omega$ there are
$k\geq n$ and $s\in\Sigma(t_n,\ldots,t_k)$ such that $\nor[s]\geq m$.    
\end{enumerate}
\end{definition}

The following two observations summarize the basic dependencies between the
notions introduced in \ref{defadd}, \ref{defbig} --- separately for the two
contexts (see \ref{rem2}).

\begin{observation}
\label{obsone}
Assume $(K,\Sigma)$ is a loose FFCC pair, $\bar{t}\in\PC(K,\Sigma)$.
\begin{enumerate}
\item If  $(K,\Sigma)$ has bigness (\ttl--additivity, respectively), then it
  has weak \ttl--bigness (weak \ttl--additivity, respectively) for the
  candidate $\bar{t}$. 
\item If $(K,\Sigma)$ has the weak \ttl--bigness for $\bar{t}$, $k<\omega$
  and $\pos(\bar{t})= \bigcup\limits_{\ell< k} \cF_\ell$, then for some
  $\bar{s}\in \PC(K,\Sigma)$ and $\ell<k$ we have  
\[\bar{t}\le \bar{s}\mbox{ and }(\forall n < \omega)(\val[s_n]\subseteq
\cF_\ell).\] 
\item Assume that $(K,\Sigma)$ has the weak \ttl--bigness property for
  $\bar{t}\in \PC(K,\Sigma)$ and it is simple except omitting. Let
  $k<\omega$ and $\pos(\bar{t})= \bigcup\limits_{\ell<k}\cF_\ell$. Then for
  some $\bar{s}\geq \bar{t}$ and $\ell<k$ we have $\pos(\bar{s}) \subseteq
  \cF_\ell$.     
\end{enumerate}
\end{observation}

\begin{observation}
\label{obsbis}
Assume $(K,\Sigma)$ is a tight FFCC pair, $\bar{t}\in\PCt(K,\Sigma)$.
\begin{enumerate}
\item If $(K,\Sigma)$ has bigness and is gluing on $\bar{t}$, then it 
  has the weak \ttt--bigness for the candidate $\bar{t}$. 
\item If $(K,\Sigma)$ has \ttt--additivity, then it has the weak
  \ttt--additivity for $\bar{t}$. 
\item If $(K,\Sigma)$ has the \ttt--multiadditivity, then it has the
  \ttt--additivity and it is gluing on $\bar{t}$.
\end{enumerate}
\end{observation}

In the following two sections we will present partition theorems for the
loose and then for the tight case. First, let us offer some easy examples to 
which the theory developed later can be applied.

\begin{example}
\label{ex1.10}
Let $\bH_1(n)=n+1$ for $n<\omega$ and let $K_1$ consist of all FP creatures
$t$ for $\bH_1$ such that
\begin{itemize}
\item $\dis[t]=(u,i,A)=(u^t,i^t,A^t)$ where $u\subseteq [m^t_\dn,m^t_\up)$,
  $i\in u$, $\emptyset\neq A\subseteq \bH_1(i)$, 
\item $\nor[t]=\log_2(|A|)$,
\item $\val[t]\subseteq\prod\limits_{j\in u}\bH_1(j)$ is such that
  $\{f(i):f\in\val[t]\}=A$.  
\end{itemize}
For $t_0,\ldots,t_n\in K_1$ with $m^{t_\ell}_\up\leq m^{t_{\ell+1}}_\dn$ let
$\Sigma_1(t_0,\ldots,t_n)$ consist of all creatures $t\in K_1$ such that 
\[m^t_\dn=m^{t_0}_\dn,\ m^t_\up=m^{t_n}_\up,\ u^t=\bigcup\limits_{\ell\leq
  n} u^{t_\ell},\ i^t=i^{t_{\ell^*}},\ A^t\subseteq A^{t_{\ell^*}}\quad
\mbox{ for some }\ell^*\leq n,\]
and $\val[t]\subseteq\big\{f_0\cup\ldots\cup f_n:(f_0,\ldots,f_n)\in
\val[t_0]\times\ldots\times\val[t_n]\big\}$. \\ 
Also, let $\Sigma^*_1$ be $\Sigma_1$ restricted to the set of those tuples
$(t_0,\ldots,t_n)$ for which $m^{t_\ell}_\up=m^{t_{\ell+1}}_\dn$ (for
$\ell<n$). Then 
\begin{itemize}
\item $(K_1,\Sigma_1)$ is a loose FFCC pair for $\bH_1$ with bigness and
\ttl--additivity,
\item $(K_1,\Sigma_1^*)$ is a tight FFCC pair for $\bH_1$ with bigness and
  \ttt--multiadditivity, and it is gluing on every $\bar{t}\in\PCt(K_1, 
  \Sigma_1^*)$. 
\end{itemize}
\end{example}

\begin{example}
Let $\bH_2(n)=2$ for $n<\omega$ and let $K_2$ consist of all FP creatures
$t$ for $\bH_2$ such that
\begin{itemize}
\item $\emptyset\neq\dis[t]\subseteq [m^t_\dn,m^t_\up)$,
\item $\emptyset\neq\val[t]\subseteq {}^{\dis[t]}2$, 
\item $\nor[t]=\log_2(|\val[t]|)$.
\end{itemize}
For $t_0,\ldots,t_n\in K_2$ with $m^{t_\ell}_\up\leq m^{t_{\ell+1}}_\dn$ let 
$\Sigma_2(t_0,\ldots,t_n)$ consist of all creatures $t\in K_2$ such that 
\[m^t_\dn=m^{t_0}_\dn,\ m^t_\up=m^{t_n}_\up,\ \dis[t]=\dis[t_{\ell^*}],\
\mbox{ and }\val[t]\subseteq\val[t_{\ell^*}]\ \mbox{ for some
}\ell^*\leq n.\] 
Then $(K_2,\Sigma_2)$ is a loose FFCC pair for $\bH_1$ which is simple
except omitting and has bigness.
\end{example}

\begin{example}
\label{ex1.12}
Let $\bH$ be as in \ref{context} and let $K_3$ consist of all FP creatures
$t$ for $\bH$ such that
\begin{itemize}
\item $\emptyset\neq\dis[t]\subseteq [m^t_\dn,m^t_\up)$,
\item $\val[t]\subseteq \{f\in\cF_\bH:\dis[t]\subseteq\dom(f) \subseteq  
  [m^t_\dn,m^t_\up)\}$ satisfies 
\[(\forall g\in\prod\limits_{i\in\dis[t]}\bH(i))(\exists f\in \val[t])(g
\subseteq f),\]  
\item $\nor[t]=\log_{957}(|\dis[t]|)$.
\end{itemize}
For $t_0,\ldots,t_n\in K_2$ with $m^{t_\ell}_\up\leq m^{t_{\ell+1}}_\dn$ let 
$\Sigma_3(t_0,\ldots,t_n)$ consist of all creatures $t\in K_3$ such that 
\begin{itemize}
\item $m^t_\dn=m^{t_0}_\dn$, $m^t_\up=m^{t_n}_\up$, $\dis[t]\subseteq
  \bigcup\limits_{\ell\leq n}\dis[t_\ell]$, and
\item if $f \in\val[t]$, then $\dom(f)\subseteq \bigcup\{[m^{t_\ell}_\dn,
  m^{t_\ell}_\up):\ell\le n\}$ and $f\rest [m^{t_\ell}_\dn,
  m^{t_\ell}_\up)\in \val[t_\ell]\cup\{\emptyset\}$ for  all $\ell\leq n$. 
\end{itemize}
Also, for $t_0,\ldots,t_n\in K_2$ with $m^{t_\ell}_\up=m^{t_{\ell+1}}_\dn$ let 
$\Sigma_3^*(t_0,\ldots,t_n)$ consist of all creatures $t\in K_3$ such that  
\begin{itemize}
\item $m^t_\dn=m^{t_0}_\dn$, $m^t_\up=m^{t_n}_\up$, $\dis[t]\subseteq
  \bigcup\limits_{\ell\leq n}\dis[t_\ell]$, and
\item if $f \in\val[t]$, then $\dom(f)\subseteq \bigcup\{[m^{t_\ell}_\dn,
  m^{t_\ell}_\up):\ell\le n\}$ and $f\rest [m^{t_\ell}_\dn,
  m^{t_\ell}_\up)\in \val[t_\ell]$ for  all $\ell\leq n$. 
\end{itemize}
Then 
\begin{itemize}
\item $(K_3,\Sigma_3)$ is a loose FFCC pair for $\bH$ with bigness and
\ttl--additivity,
\item $(K_3,\Sigma_3^*)$ is a tight FFCC pair for $\bH$ with bigness and
\ttt--multiadditivity and it is gluing on every $\bar{t}\in\PCt(K_3, 
\Sigma_3^*)$.  
\end{itemize}
\end{example}

\begin{example}
\label{CarlSim}
Let $N>0$ and $\bH_N(n)=N$. Let $K_N$ consist of all FP creatures $t$ for
$\bH_N$ such that 
\begin{itemize}
\item $\dis[t]=(X_t,\varphi_t)$, where $X_t\subsetneq [m^t_\dn, m^t_\up)$,
and $\phi_t:X_t\longrightarrow N$,
\item $\nor[t]=m^t_\up$,
\item $\val[t]=\big\{f\in {}^{[m^t_\dn,m^t_\up)} N:\varphi_t\subseteq f$ and
  $f$ is constant on $[m^t_\dn,m^t_\up)\setminus X_t\;\big\}$.
\end{itemize}
For $t_0,\ldots,t_n\in K_N$ with $m^{t_\ell}_\up=m^{t_{\ell+1}}_\dn$ (for
$\ell<n$) we let $\Sigma_N(t_0,\ldots,t_n)$ consist of all creatures $t\in
K_N$ such that 
\begin{itemize}
\item $m^t_\dn=m^{t_0}_\dn$, $m^t_\up=m^{t_0}_\up$, $X_{t_0}\cup\ldots\cup
X_{t_n} \subseteq X_t$,
\item for each $\ell\leq n$,\\
either $X_t\cap [m^{t_\ell}_\dn,m^{t_\ell}_\up)=X_{t_\ell}$ and
$\varphi_t\rest [m^{t_\ell}_\dn,m^{t_\ell}_\up)=\varphi_{t_\ell}$,\\
or $[m^{t_\ell}_\dn,m^{t_\ell}_\up)\subseteq X_t$ and $\varphi_t\rest
[m^{t_\ell}_\dn,m^{t_\ell}_\up)\in \val[t_\ell]$. 
\end{itemize}
Then 
\begin{enumerate}
\item[(i)] $(K_N,\Sigma_N)$ is a tight FFCC pair for $\bH_N$,
\item[(ii)] it has the \ttt--multiadditivity and 
\item[(iii)] it has the weak \ttt--bigness and is gluing for every candidate
  $\bar{t}\in\PCt(K,\Sigma)$. 
\end{enumerate}
\end{example}

\begin{proof}
(i)\quad All demands in \ref{g1}(2,4) are easy to verify. For instance, to
check \ref{g1}(4)(h$^{\rm tight}$) note that:\\
if $s_0,s_1\in K_N$, $m^{s_0}_\up=m^{s_1}_\dn$, $f_\ell\in \val[s_\ell]$
(for $\ell=0,1$) and $s\in K_N$ is such that
\[m^s_\dn=m^{s_0}_\dn,\ m^s_\up=m^{s_1}_\up,\ X_s=X_{s_0}\cup
[m^{s_1}_\dn,m^{s_1}_\up),\ \varphi_s\rest X_{s_0}=\varphi_{s_0},\
\varphi_s\rest [m^{s_1}_\dn,m^{s_1}_\up)=f_1,\]
then $s\in \Sigma_N(s_0,s_1)$ and $f_0\cup f_1\in \val[s]$.

\noindent (ii)\quad The $s$ constructed as in (i) above for $s_0,s_1$ will
witness the \ttt--additivity as well. In an analogous way we show also the
multiadditivity. 

\noindent (iii)\quad Let $\bar{t}=\langle t_i:i<\omega\rangle\in
\PCt(K_N,\Sigma_N)$. Suppose that $n,n_1,n_2<\omega$ and $\post(\bar{t}\upl
n)=\bigcup\{\cF_\ell:\ell<n_1\}$. By the Hales--Jewett theorem (see
\cite{HJ63}) there is $k>n_2$ such that for any partition of ${}^kN$ into
$n_1$ parts there is a combinatorial line included in one of the parts. Then
we easily find $s\in\Sigma_N(t_0,\ldots,t_{k-1})$ such that $\val[s]
\subseteq \cF_\ell$ for some $\ell<n_1$. Necessarily, $\nor[s]\geq k-1\geq
n_2$. This proves the weak \ttt--bigness for $\bar{t}$. Similarly to (ii) we
may argue that $(K_N,\Sigma_N)$ is gluing on $\bar{t}$.  
\end{proof}

\section{Ultrafilters on loose possibilities}
Here we introduce ultrafilters on the (countable) set $\cF_\bH$ (see
\ref{context}) which contain sets large from the point of view of pure
candidates for a loose FFCC pair. Then we use them to derive a partition
theorem for this case.

\begin{definition}  
\label{g5}
Let $(K,\Sigma)$ be a loose FFCC pair for $\bH$. 
\begin{enumerate}
\item For a pure candidate $\bar{t}\in\PC(K,\Sigma)$, we  define 
\begin{itemize}
\item $\cA_{\bar{t}}^0=\{\pos(\bar{t}\upl n): n<\omega\}$, 
\item $\cA_{\bar{t}}^1$ is the collection of all sets $A\subseteq \cF_\bH$
  such that for some $N<\omega$ we have 
\[\big(\forall s\in\Sigma'(\bar{t})\big)\big(\nor[s]\ge N\ \&\ m^s_\dn\ge
N\quad\Rightarrow \quad \val[s]\cap A\neq \emptyset\big),\]
\item $\cA_{\bar{t}}^2$ is the collection of all sets $A\subseteq \cF_\bH$
such that for some $N<\omega$ we have 
\[(\forall \bar{t}_1\geq\bar{t})(\exists \bar{t}_2\geq \bar{t}_1)(\forall
s\in \Sigma'(\bar{t}_2))(\nor[s]\ge N \quad\Rightarrow \quad \val[s]\cap
A\neq \emptyset\big).\] 
\end{itemize}
\item For $\ell<3$ we let $\uf^\ell_{\bar{t}}(K,\Sigma)$ be the family of
  all ultrafilters $D$ on $\cF_\bH$ such that $\cA_{\bar{t}}^\ell\subseteq
  D$. We also set (for $\ell<3$) 
\[\uf_\ell(K,\Sigma)\stackrel{\rm def}{=}\bigcup\{\uf^\ell_{\bar{t}}(K,
\Sigma):\bar{t}\in \PC(K,\Sigma)\}.\]
\end{enumerate}
\end{definition}

\begin{proposition}
\label{g7d}
Let $(K,\Sigma)$ be a loose FFCC pair for $\bH$, $\bar{t}\in\PC(K,\Sigma)$.   
\begin{enumerate}
\item $\cA^0_{\bar{t}}\subseteq\cA^1_{\bar{t}}\subseteq \cA^2_{\bar{t}}$ and
  hence also $\uf^2_{\bar{t}}(K,\Sigma) \subseteq \uf^1_{\bar{t}}(K,\Sigma) 
  \subseteq \uf^0_{\bar{t}}(K,\Sigma)$.
\item $\uf^0_{\bar{t}}(K,\Sigma) \neq\emptyset$.
\item If $(K,\Sigma)$ has the weak \ttl--bigness for each
  $\bar{t}'\geq\bar{t}$, then $\uf^2_{\bar{{t}}}(K,\Sigma)$ is not empty. 
\item If $(K,\Sigma)$ has the weak \ttl--bigness for $\bar{t}$, then
  $\uf^1_{\bar{{t}}}(K,\Sigma)\neq \emptyset$.
\item Assume {\rm CH}. Suppose that $(K,\Sigma)$ is simple except omitting
  (see \ref{simple}(1)) and has the weak \ttl--bigness on every candidate 
  $\bar{t}\in \PC(K,\Sigma)$. Then there is $D\in\uf^2_{\bar{{t}}}
  (K,\Sigma)$ such that  
\[\big(\forall A\in D\big)\big(\exists \bar{t}\in\PC(K,\Sigma)\big)
\big(\pos(\bar{t})\in D\ \&\ \pos(\bar{t})\subseteq A\big).\] 
\end{enumerate}
\end{proposition}

\begin{proof}
(2)\quad Note that $\cA^0_{\bar{t}}$ has the finite intersection property
(fip).  
\medskip

\noindent (3)\quad It is enough to show that, assuming $(K,\Sigma)$ has the
weak \ttl--bigness for all $\bar{t}'\geq \bar{t}$, $\cA^2_{\bar{t}}$ has
fip. So suppose that for $\ell<k$ we are given a set
$A_\ell\in\cA^2_{\bar{t}}$ and let $N_\ell<\omega$ be such that 
\begin{enumerate}
\item[$(*)_\ell$] $(\forall \bar{t}_1\geq\bar{t})(\exists \bar{t}_2\geq
  \bar{t}_1)(\forall s\in \Sigma'(\bar{t}_2))(\nor[s]\ge N_\ell \
  \Rightarrow \ \val[s]\cap A_\ell\neq \emptyset\big)$.  
\end{enumerate}
Let $N=\max\{N_\ell:\ell<k\}$. Then we may choose $\bar{t}'\geq \bar{t}$
such that 
\begin{enumerate}
\item[$(*)$] $(\forall s\in \Sigma'(\bar{t}'))(\nor[s]\ge N\ \Rightarrow \
  (\forall \ell<k)(\val[s]\cap A_\ell\neq \emptyset)\big)$.  
\end{enumerate}
[Why? Just use repeatedly  $(*)_\ell$ for $\ell=0,1,\ldots,k-1$; remember 
$\bar{t}'\leq\bar{t}''$ implies $\Sigma'(\bar{t}'')\subseteq
\Sigma'(\bar{t}')$.]   

For $\eta\in {}^k2$ set 
\[\cF_\eta=\{f\in\pos(\bar{t}'):(\forall\ell<k)(\eta(\ell)=1\
\Leftrightarrow\ f\in A_\ell)\}.\] 
Then $\pos(\bar{t}')=\bigcup\{\cF_\eta:\eta\in {}^k 2\}$ and $(K,\Sigma)$
has the weak $\ttl$--bigness for $\bar{t}'$, so we may use Observation
\ref{obsone}(2) to pick $\eta_0\in {}^k 2$ and $\bar{s}\geq \bar{t}'$ such
that $\val[s_n]\subseteq\cF_{\eta_0}$ for all $n<\omega$. Consider
$n<\omega$ such that $\nor[s_n]>N$. It follows from $(*)$ that
$\val[s_n]\cap A_\ell\neq\emptyset$ for all $\ell<k$. Hence, by the choice
of $\bar{s}$, $\eta_0(\ell)=1$ for all $\ell<k$ and therefore
$\emptyset\neq\val[s_n]\subseteq \bigcap\limits_{\ell<k} A_\ell$.  \medskip

\noindent (4)\quad Similarly to (3) above one shows that $\cA^1_{\bar{t}}$
has fip.

\noindent (5)\quad Assuming CH  and using Observation \ref{obsone}(3) we may
construct a sequence $\langle \bar{t}_\alpha:\alpha<\omega_1\rangle\subseteq
\PC(K,\Sigma)$ such that 
\begin{itemize}
\item if $\alpha<\beta<\omega_1$ then $(\exists n<\omega)(\bar{t}_\alpha
  \leq (\bar{t}_\beta\upl n))$, 
\item if $A\subseteq \cF_\bH$ then for some $\alpha<\omega_1$ we have that 
  either $\pos(\bar{t}_\alpha)\subseteq A$ or $\pos(\bar{t}_\alpha)\cap
  A=\emptyset$. 
\end{itemize} 
(Compare to the proof of \cite[5.3.4]{RoSh:470}.) Then the family 
\[\{\pos(\bar{t}_\alpha\upl n):\alpha<\omega_1\ \&\ n<\omega\}\] 
generates the desired ultrafilter.
\end{proof}

\begin{observation}
\label{obs}
The sets $\uf^\ell_{\bar{t}}(K,\Sigma)$ (for $\ell<3$) are closed subsets of
the (Hausdorff compact topological space) $\beta_*(\cF_\bH)$ of
non-principal ultrafilters on $\cF_\bH$. Hence each
$\uf^\ell_{\bar{t}}(K,\Sigma)$ itself is a compact Hausdorff space. 
\end{observation}

\begin{definition}
\label{g6}
\begin{enumerate}
\item For $f\in\cF_\bH$ and $A\subseteq\cF_\bH$ we define 
\[f\oplus A\stackrel{\rm def}{=}\{g\in\cF_\bH:\max(\dom(f))<\min(\dom(g))\
\mbox{ and }\ f\cup g \in A\}.\] 
\item For $D_1,D_2\in\uf_0(K,\Sigma)$ we let 
\[D_1 \oplus D_2 \stackrel{\rm def}{=}\big\{A\subseteq \cF_\bH:\{f \in
\cF_\bH: (f \oplus A) \in D_1\} \in D_2\big\}.\] 
\end{enumerate}
\end{definition} 

\begin{proposition}
\label{obstwo}
\begin{enumerate}
\item If $A_1,A_2\subseteq\cF_\bH$ and $f\in\cF_\bH$, then
\[\begin{array}{l}
f\oplus(A_1\cap A_2)=(f\oplus A_1)\cap (f\oplus A_2)\quad\mbox{ and}\\
\{g\in\cF_\bH:\max(\dom(f))<\min(\dom(g))\}\setminus (f\oplus A_1)=
f\oplus (\cF_\bH\setminus A_1).
\end{array}\]
\item If $D_1,D_2,D_3\in\uf_0(K,\Sigma)$, then $D_1\oplus D_2$ is a
  non-principal ultrafilter on $\cF_\bH$ and $D_1\oplus(D_2\oplus D_3)=
  (D_1\oplus D_2)\oplus D_3$. 
\item The mapping $\oplus: \uf_0(K,\Sigma)\times \uf_0(K,\Sigma)
\longrightarrow \beta_*(\cF_\bH)$ is right continuous (i.e., for each $D_1
\in \uf_0(K,\Sigma)$ the function $\uf_0(K,\Sigma)\ni D_2 \mapsto D_1 \oplus
D_2 \in \beta_*(\cF_\bH)$ is continuous).
\end{enumerate}
\end{proposition}

\begin{proof}
  Straightforward, compare with \ref{correctop}.
\end{proof}

\begin{proposition}
\label{g7} 
Assume that a loose FFCC pair $(K,\Sigma)$ has the weak \ttl--additivity
(see \ref{defadd}(1)) for a candidate $\bar{t}\in \PC(K,\Sigma)$. If
$D_1,D_2\in \uf^1_{\bar{t}}(K,\Sigma)$, then $D_1\oplus D_2\in
\uf^1_{\bar{t}}(K,\Sigma)$.     
\end{proposition} 

\begin{proof}
Let $\bfun:\omega \rightarrow \omega$ witness the weak \ttl--additivity of 
$(K,\Sigma)$ for $\bar{t}$, and let $D = D_1 \oplus D_2$, $D_1,D_2 \in  
\uf^1_{\bar{t}}(K,\Sigma)$. We already know that $D$ is an ultrafilter on
$\cF_\bH$ (by \ref{obstwo}(2)), so we only need to show that it includes 
$\cA^1_{\bar{t}}$. 

Suppose that $A\in \cA^1_{\bar{t}}$ and let $N<\omega$ be such that  
\begin{enumerate}
\item[$(*)_1$] $\big(\forall s\in\Sigma'(\bar{t})\big)\big(\nor[s]\ge N\ \&\
  m^s_\dn\ge N\quad\Rightarrow \quad \val[s]\cap A\neq \emptyset\big)$.
\end{enumerate}

\begin{claim}
\label{cl1}
For every $s\in\Sigma'(\bar{t})$, if $\nor[s]\ge \bfun(N)$ and $m^s_\dn\ge 
\bfun(N)$, then $\val[s]\cap \{f\in\cF_\bH:f\oplus A\in D_1\}\neq
\emptyset$.  
\end{claim}

\begin{proof}[Proof of the Claim]
Suppose $s_0\in\Sigma'(\bar{t})$, $\nor[s_0]\ge \bfun(N)$, $m^{s_0}_\dn\ge 
\bfun(N)$. Set    
\[B=\bigcup\{f\oplus A:f\in \val[s_0]\}.\]
We are going to argue that 
\begin{enumerate}
\item[$(*)_2$] $B\in\cA^1_{\bar{t}}$.
\end{enumerate}
So let $M=\bfun(m^{s_0}_\up)+m^{s_0}_\up+957$ and suppose
$s_1\in\Sigma(\bar{t})$ is such that $\nor[s_1]\geq M$ and $m^{s_1}_\dn\geq
M$. Apply the weak additivity and the choice of $M$ to find $s\in 
\Sigma'(\bar{t})$ such that    
\[m^s_\dn\geq N,\quad \nor[s]\geq N\quad \mbox{ and }\quad \val[s]\subseteq
\{f\cup g:f\in \val[s_0]\ \&\ g\in\val[s_1]\}.\]
Then, by $(*)_1$, $\val[s]\cap A\neq \emptyset$ so for some $f\in\val[s_0]$
and $g\in\val[s_1]$ we have $f\cup g\in A$ (and
$\max(\dom(f))<\min(\dom(g))$). Thus $g\in (f\oplus A)\cap
\val[s_1]\subseteq B\cap\val[s_1]$ and $(*)_2$ follows. 

Since $D_1\in\uf^1_{\bar{t}}(K,\Sigma)$ we conclude from $(*)_2$ that $B\in
D_1$ and hence (as $\val[s_0]$ is finite) $f\oplus A\in D_1$ for some
$f\in\val[s_0]$, as desired.
\end{proof} 

It follows from \ref{cl1} that $\{f\in\cF_\bH: f\oplus A\in
D_1\}\in\cA^1_{\bar{t}}$ and hence (as $D_2\in\uf^1_{\bar{t}}(K,\Sigma)$)
$\{f\in\cF_\bH: f\oplus A\in D_1\}\in D_2$. Consequently, $A\in D_1\oplus
D_2$.    
\end{proof}

\begin{lemma}
[See {\cite[Lemma 10.1, p.449]{Cmf77}}]
\label{z4}
If $X$ is a non-empty compact Hausdorff space, $\odot$ an associative binary 
operation which is continuous from the right (i.e. for each $p \in X$
the function $q \mapsto p \odot q$ is continuous), then there is a
$\odot$--idempotent point $p \in X$ (i.e. $p \odot p = p$).
\end{lemma}

\begin{corollary}
\label{g8}
Assume that a loose FFCC pair $(K,\Sigma)$ has weak \ttl--additivity and
the weak \ttl--bigness for a candidate $\bar{t} \in \PC(K,\Sigma)$. Then 
\begin{enumerate}
\item $\uf^1_{\bar{t}}(K,\Sigma)$ is a non-empty compact Hausdorff space and
  $\oplus$ is an associative right continuous operation on it.
\item There is $D \in \uf^1_{\bar{t}}(K,\Sigma)$ such that $D = D \oplus D$.
\end{enumerate}
\end{corollary} 

\begin{proof}
(1)\quad  By \ref{g7d}(3), \ref{obs}, \ref{obstwo}(2,3) and \ref{g7}. 
\medskip

(2)\quad It follows from (1) above that all the assumptions of Lemma
\ref{z4} are satisfied for $\oplus$ and $\uf^1_{\bar{t}}(K,\Sigma)$, hence
its conclusion holds.  
\end{proof}

\begin{theorem}
\label{g9}
Assume that $(K,\Sigma)$ is a loose FFCC pair, $\bar{t}\in 
\PC(K,\Sigma)$. Let an ultrafilter $D\in\uf^1_{\bar{t}}(K,\Sigma)$ be such
that $D \oplus D = D$. Then  
\[\big(\forall A\in D\big)\big(\exists \bar{s}\geq \bar{t}\big)
\big(\pos(\bar{s})\subseteq A\big).\]
\end{theorem} 

\begin{proof}
The main ingredient of our argument is given by the following claim. 

\begin{claim}
\label{cl2}
Let $(K,\Sigma)$, $\bar{t}$ and $D$ be as in the assumptions of
\ref{g9}. Assume $A\in D$ and $n<\omega$. Then there is
$s\in\Sigma'(\bar{t})$ such that 
\begin{enumerate}
\item[$(\bullet)_1$]   $\val[s]\subseteq A$, $\nor[s]\geq n$, $m^s_\dn\geq
  n$, and 
\item[$(\bullet)_2$]   $(\forall f\in\val[s])(f\oplus A\in D)$. 
\end{enumerate}
\end{claim}

\begin{proof}[Proof of the Claim]
Let $A' := \{f \in\cF_\bH:f \oplus A\in D\}$ and $A'':= A\cap A'$. Since $A
\in D= D\oplus D$ we know that $A'\in D$ and thus $A''\in D$. Hence
$\cF_\bH\setminus A''\notin \cA^1_{\bar{t}}$ (remember $D\in
\uf^1_{\bar{t}}(K,\Sigma)$). Therefore, there is $s\in \Sigma'(\bar{t})$
such that   
\[\nor[s]\geq n,\quad m^s_\dn\geq n,\quad \mbox{ and }\quad \val[s]\cap
(\cF_\bH\setminus A'')=\emptyset.\] 
Then $\val[s]\subseteq A$ and for each $f\in\val[s]$ we have $f\oplus
A\in D$, as desired. 
\end{proof}

Now suppose $A\in D$. By induction on $n$ we choose $s_n,A_n$ so that  
\begin{enumerate}
\item[(a)] $A_0=A$, $A_n\in D$ and $A_{n+1}\subseteq A_n$, 
\item[(b)] $s_n\in \Sigma'(\bar{t})$, $\nor(s_n) \ge n$ and $m^{s_n}_\up\leq
  m^{s_{n+1}}_\dn$, 
\item[(c)] $\val[s_n]\subseteq A_n$,
\item[(d)] if $f\in A_{n+1}$, then $m^{s_n}_\up\leq \min(\dom(f))$,
\item[(e)] if $f \in \val[s_n]$, then $A_{n+1}\subseteq f\oplus A_n$.  
\end{enumerate}
Suppose we have constructed $s_0,\ldots,s_{n-1}$ and $A_n$ so that
demands (a)--(e) are satisfied. Set $N=m^{s_{n-1}}_\up+n+1$ (if $n=0$
stipulate $m^{s_{n-1}}_\up=0$) and use Claim \ref{cl2} to find $s_n\in 
\Sigma'(\bar{t})$ such that  
\begin{enumerate}
\item[$(\bullet)_1^n$] $\val[s_n]\subseteq A_n$, $\nor[s_n]\geq N$,
  $m^{s_n}_\dn\geq N$, and  
\item[$(\bullet)_2^n$] $(\forall f\in\val[s_n])(f\oplus A_n \in D)$.  
\end{enumerate}
Put 
\[A_{n+1}:=A_n\cap \{g\in\cF_\bH: m^{s_n}_\up<\min(\dom(g))\}\cap
\bigcap_{f\in\val[s_n]} f\oplus A_n.\]  
Since $\cA^0_{\bar{t}}\subseteq D$ we know that
$\{g\in\cF_\bH:m^{s_n}_\up<\min(\dom(g))\}\in D$ and since $\val[s_n]$ is
finite (and by $(\bullet)_2^n$) also $\bigcap\limits_{f\in\val[s_n]}f\oplus
A_n\in D$. Thus $A_{n+1}\in D$. Plainly the other requirements hold too.

After the above construction is carried out we set $\bar{s}= \langle
s_n:n<\omega\rangle$. Clearly $\bar{s}\in\PC(K,\Sigma)$ and
$\bar{s}\geq\bar{t}$ (remember clause (b)). 

\begin{claim}
\label{cl3}
If $n_0<\ldots<n_k<\omega$ and $f_\ell\in\val[s_{n_\ell}]$ for $\ell\le k$,
then $\bigcup\limits_{\ell\leq k} f_\ell\in A_{n_0}$.
\end{claim}

\begin{proof}[Proof of the Claim]
Induction on $k$. If $k=0$ then clause (c) of the choice of $s_{n_0}$ gives
the conclusion. For the inductive step suppose the claim holds true for $k$
and let $n_0<n_1<\ldots<n_k<n_{k+1}$, $f_\ell\in\val[s_{n_\ell}]$ (for
$\ell\leq k+1$). Letting $g=f_1\cup\ldots\cup f_{k+1}$ we may use the
inductive hypothesis to conclude that $g\in A_{n_1}$. By (a)+(e) we know
that $A_{n_1}\subseteq A_{n_0+1}\subseteq f_0\oplus A_{n_0}$, so $g\in
f_0 \oplus A_{n_0}$. Hence $f_0\cup g=f_0\cup f_1\cup\ldots\cup f_{k+1}\in
A_{n_0}$. 
\end{proof}
It follows from \ref{cl3} that $\pos(\bar{s})\subseteq A$ (remember (a)
above and \ref{g1}(2)(f)).  
\end{proof} 

\begin{conclusion}
\label{g11}
Suppose that  $(K,\Sigma)$ is a loose FFCC pair with weak \ttl--bigness and
weak \ttl--additivity over $\bar{t} \in \PC(K,\Sigma)$. Assume also that
$\pos(\bar{t})$ is the finite union $\cF_0 \cup \ldots \cup\cF_n$.  Then for
some $i\le n$ and $\bar{s} \in \PC(K,\Sigma)$ we have 
\[\pos(\bar{s}) \subseteq \cF_i\quad\mbox{ and }\quad\bar{t}\le\bar{s}.\]  
\end{conclusion} 

\begin{proof}
By \ref{g8} there is $D\in\uf^1_{\bar{t}}(K,\Sigma)$ such that $D = D \oplus
D$.  Clearly for some $i \le n$ we have $\cF_i \in D$. By \ref{g9} there
is $\bar{s} \in \PC(K,\Sigma)$ such that $\bar{t} \le\bar{s}$ and
$\pos(\bar{s}) \subseteq \cF_i$. 
\end{proof}

\section{Ultrafilters on tight possibilities}
In this section we carry out for tight FFCC pairs considerations parallel to
that from the case of loose FFCC pairs. The main difference now is that we
use {\em sequences\/} of ultrafilters, but many arguments do not change
much. 

\begin{definition}
\label{sufs}
Let $(K,\Sigma)$ be a tight FFCC pair for $\bH$, $\bar{t}=\langle
t_n:n<\omega\rangle\in\PCt(K,\Sigma)$.  
\begin{enumerate}
\item For $f\in\post(\bar{t}\upl n)$ let $x_f=x_f^{\bar{t}}$ be
  the unique $m>n$ such that $f\in\val[s]$ for some $s\in\Sigma(
  t_n,\ldots,t_{m-1})$. (Note \ref{g1}(4)(f$^{\rm tight}$).)  
\item If $f\in\post(\bar{t}\upl n)$, $n<\omega$, $A\subseteq
  \cF_\bH$, then we set 
\[f\circledast A=f\circledast_{\bar{t}} A=\big\{g\in \post(\bar{t}
\upl x_f):f\cup g\in A\big\}.\]
\item We let $\suft$ be the set of all sequences $\bar{D}=\langle
  D_n:n<\omega\rangle$ such that each $D_n$ is a non-principal ultrafilter
  on $\post(\bar{t}\upl n)$.
\item The space $\suft$ is equipped with the (Tichonov) product topology of
  $\prod\limits_{n<\omega}\beta_*\big(\post(\bar{t}\upl
  n)\big)$. For a sequence $\bar{A}=\langle A_0,\ldots,A_n\rangle$ such that
  $A_\ell\subseteq \post(\bar{t}\upl \ell)$ (for $\ell\leq n$) we set 
\[{\rm Nb}_{\bar{A}}=\big\{\bar{D}\in\suft:(\forall \ell\leq n)( A_\ell\in
D_\ell)\big\}.\] 
\item For $\bar{D}=\langle D_n:n<\omega\rangle\in\suft$, $n<\omega$ and
  $A\subseteq\post(\bar{t}\upl n)$ we let
\[\set^n_{\bar{t}}(A,\bar{D})=\big\{f\in\post(\bar{t}\upl n):
f\circledast A\in D_{x_f}\big\}.\]
\item For $\bar{D}^1,\bar{D}^2\in\suft$ we define $\bar{D}^1\circledast
  \bar{D}^2$ to be a sequence $\langle D_n:n<\omega\rangle$ such that for
  each $n$ 
\[D_n=\big\{A\subseteq\post(\bar{t}\upl n): \set^n_{\bar{t}}
(A,\bar{D}^1)\in D^2_n\big\}.\]
\end{enumerate}
\end{definition}

\begin{observation}
\label{obsforassoc}
Let $(K,\Sigma)$ be a tight FFCC pair for $\bH$ and $\bar{t}\in\PCt
(K,\Sigma)$. Suppose $f\in\post(\bar{t}\upl n)$, $g\in\post(
\bar{t}\upl x_f)$. Then
\begin{enumerate}
\item $f\cup g\in\post(\bar{t}\upl n)$ (note \ref{g1}(4)(h$^{\rm
    tight}$)) and
\item $(f\cup g)\circledast A=g\circledast (f\circledast A)$ for all
  $A\subseteq \cF_\bH$.   
\item $\suft$ is a compact Hausdorff topological space. The sets ${\rm
    Nb}_{\bar{A}}$ for $\bar{A}=\langle A_0,\ldots,A_n\rangle$, $A_\ell
  \subseteq \post(\bar{t}\upl \ell)$, $\ell\leq n<\omega$, form a
  basis of the topology of $\suft$.
\end{enumerate}   
\end{observation}

\begin{proposition}
\label{correctop}
Let $(K,\Sigma)$ be a tight FFCC pair for $\bH$ and $\bar{t}\in\PCt
(K,\Sigma)$.
\begin{enumerate}
\item If $\bar{D}^1,\bar{D}^2\in \suft$, then $\bar{D}^1\circledast
  \bar{D}^2 \in\suft$.
\item The mapping $\circledast:\suft\times\suft\longrightarrow \suft$ is
  right continuous.
\item The operation $\circledast$ is associative.
\end{enumerate}
\end{proposition}

\begin{proof}
(1)\qquad Let $\bar{D}^1,\bar{D}^2\in \suft$, $n<\omega$, and
\[D_n=\big\{A\subseteq\post(\bar{t}\upl n):\set^n_{\bar{t}}(
A,\bar{D}^1)\in \bar{D}^2_n\big\}.\]
Let $A,B\subseteq\post(\bar{t}\upl n)$.

\noindent (a)\quad  If $f\in\post(\bar{t}\upl n)$ and $A$ is
finite, then $f\circledast A$ is finite as well, so it does not belong to
$D^1_{x_f}$. Consequently, if $A$ is finite then $\set^n_{\bar{t}} (A,
\bar{D}^1)=\emptyset$ and $A\notin D_n$.  

\noindent (b)\quad  $\set^n_{\bar{t}}(\post(\bar{t}\upl n),
\bar{D}^1)=\post(\bar{t}\upl n)\in D^2_n$ (note
\ref{obsforassoc}(1)). Thus $\post(\bar{t}\upl n)\in D_n$. 

\noindent (c)\quad  If $A\subseteq B$ then $\set^n_{\bar{t}}(A,\bar{D}^1)
\subseteq \set^n_{\bar{t}}(B,\bar{D}^1)$ and hence 
\[A\subseteq B\ \&\ A\in D_n\ \ \Rightarrow\ \ B\in D_n.\]

\noindent (d)\quad  $\set^n_{\bar{t}}(A\cap B,\bar{D}^1)=
\set^n_{\bar{t}}(A,\bar{D}^1)\cap \set^n_{\bar{t}}(B,\bar{D}^1)$ and hence 
\[A,B\in D_n\ \ \Rightarrow\ \ A\cap B\in D_n.\] 

\noindent (e)\quad  $\set^n_{\bar{t}}(\post(\bar{t}\upl n)
\setminus A,\bar{D}^1)=\post(\bar{t}\upl n)\setminus
\set^n_{\bar{t}}(A,\bar{D}^1)$, and hence 
\[A\notin D_n\ \ \Rightarrow\ \ \post(\bar{t}\upl n)\setminus A\in
D_n.\] 

It follows from (a)--(e) that $D_n$ is a non-principal ultrafilter on
$\post(\bar{t}\upl n)$ and hence clearly $\bar{D}^1\circledast
\bar{D}^2\in\suft$.

(2)\qquad Fix $\bar{D}^1\in\suft$ and let $\bar{A}=\langle A_\ell:\ell\leq n
\rangle$, $A_\ell\subseteq\post(\bar{t}\upl \ell)$. For $\ell\leq
n$ put $B_\ell=\set^\ell_{\bar{t}}(A_\ell,\bar{D}^1)$ and let
$\bar{B}=\langle B_\ell:\ell\leq n\rangle$. Then for each $\bar{D}^2\in
\suft$ we have
\[\bar{D}^1\circledast\bar{D}^2\in {\rm Nb}_{\bar{A}}\ \ \mbox{ if and only
  if }\ \ \bar{D}^2\in {\rm Nb}_{\bar{B}}.\]

(3)\qquad Let $\bar{D}^1,\bar{D}^2,\bar{D}^3\in \suft$. Suppose $n<\omega$,
$A\subseteq \pos(\bar{t}\upl n)$. Then 
\begin{enumerate}
\item[(i)] $A\in \big((\bar{D}^1\circledast\bar{D}^2)\circledast
  \bar{D}^3\big)_n$\ \ iff\ \ $\set^n_{\bar{t}}(A,\bar{D}^1\circledast
  \bar{D}^2)\in D^3_n$\ \ iff\\
$\big\{f\in\post(\bar{t}\upl n): f\circledast A\in
(\bar{D}^1\circledast \bar{D}^2)_{x_f}\big\}\in D^3_n$, and 
\item[(ii)] $A\in\big(\bar{D}^1\circledast (\bar{D}^2\circledast\bar{D}^3)  
  \big)_n$\ \ iff\ \ $\set^n_{\bar{t}}(A,\bar{D}^1)\in (\bar{D}^2\circledast 
  \bar{D}^3)_n$\ \ iff\\ 
$\set^n_{\bar{t}}\big(\set^n_{\bar{t}}(A,\bar{D}^1),\bar{D}^2\big)\in
\bar{D}^3_n$\ \ iff\ \ $\big\{f\in\post(\bar{t}\upl n): f\circledast
\set^n_{\bar{t}}(A,\bar{D}^1)\in D^2_{x_f}\big\}\in D^3_n$. 
\end{enumerate}
Let us fix $f\in\post(\bar{t}\upl n)$ for a moment. Then 
\[\begin{array}{l}
f\circledast A\in \big(\bar{D}^1\circledast \bar{D}^2\big)_{x_f}\mbox{ iff } 
\set^{x_f}_{\bar{t}}(f\circledast A,\bar{D}^1)\in D^2_{x_f} \mbox{ iff}\\
\big\{g\in\post(\bar{t}\upl x_f): g\circledast (f\circledast A)
\in D^1_{x_g}\big\}\in D^2_{x_f}\mbox{ iff}\\
\big\{g\in\post(\bar{t}\upl x_f): (f\cup g)\circledast A\in
D^1_{x_g}\big\}\in D^2_{x_f}\mbox{ iff}\\
\big\{g\in \post(\bar{t}\upl x_f): (f\cup g)\in
\set^n_{\bar{t}}(A,\bar{D}^1) \big\}\in D^2_{x_f}\mbox{ iff } f\circledast
\set^n_{\bar{t}}(A,\bar{D}^1)\in D^2_{x_f}.
\end{array}\]
Consequently,
\[\big\{f\in\post(\bar{t}{\upl} n):f\circledast A\in (\bar{D}^1 
\circledast \bar{D}^2)_{x_f}\big\}=\big\{f\in\post(\bar{t} {\upl}
n): f\circledast \set^n_{\bar{t}}(A,\bar{D}^1)\in D^2_{x_f}\big\}\]
and (by (i)+(ii)) $A\in \big(\bar{D}^1\circledast (\bar{D}^2 \circledast
\bar{D}^3)\big)_n$ if and only if $A\in \big((\bar{D}^1\circledast
\bar{D}^2) \circledast \bar{D}^3\big)_n$. 
\end{proof}

\begin{definition}
\label{sufstar}
Let $(K,\Sigma)$ be a tight FFCC pair for $\bH$ and $\bar{t}\in\PCt
(K,\Sigma)$.
\begin{enumerate}
\item For $n<\omega$, $\cB^n_{\bar{t}}$ is the family of all sets
  $B\subseteq \post(\bar{t}\upl n)$ such that for some $M$ we
  have:\\
if $s\in\Sigt(\bar{t}\upl n)$ and $\nor[s]\geq M$, then $\val[s]\cap 
B\neq\emptyset$. 
\item $\sufs$ is the family of all $\bar{D}=\langle D_n:n<\omega\rangle
  \in\suft$ such that $\cB^n_{\bar{t}}\subseteq D_n$ for all $n<\omega$.  
\end{enumerate}
\end{definition}

\begin{proposition}
\label{n19}
Let $(K,\Sigma)$ be a tight FFCC pair for $\bH$ and $\bar{t}\in\PCt
(K,\Sigma)$.
\begin{enumerate}
\item $\sufs$ is a closed subset of $\suft$.
\item If $(K,\Sigma)$ has the weak \ttt--bigness for $\bar{t}$, then
  $\sufs\neq\emptyset$. 
\item If $(K,\Sigma)$ has the weak \ttt--additivity for $\bar{t}$, then
  $\sufs$ is closed under $\circledast$.
\end{enumerate}
\end{proposition}

\begin{proof}
(1)\qquad Suppose $\bar{D}\in\suft\setminus\sufs$. Let $n<\omega$ and
$B\in\cB^n_{\bar{t}}$ be such that $B\notin D_n$. Set $A_n=\post(\bar{t}\upl
n)\setminus B$ and $A_\ell=\post(\bar{t}\upl \ell)$ for $\ell<n$, and let 
$\bar{A}=\langle A_0,\ldots,A_n\rangle$. Then $\bar{D}\in {\rm Nb}_{\bar{A}}
\subseteq \suft\setminus\sufs$. 

\noindent (2)\qquad It is enough to show that, assuming the weak
\ttt--bigness, each family $\cB^n_{\bar{t}}$ has fip. To this end suppose
that $B_0,\ldots, B_{m-1}\in \cB^n_{\bar{t}}$. Pick $M_0$ such that 
\begin{enumerate}
\item[$(*)$] $\big(\forall s\in\Sigt(\bar{t}\upl n)\big)
  \big(\forall \ell<m\big)\big(\nor[s]\geq M_0\ \Rightarrow\ B_\ell\cap 
  \val[s]\neq \emptyset\big)$.
\end{enumerate}
For $\eta\in {}^m2$ set $C_\eta=\{f\in\post(\bar{t}\upl n):
(\forall \ell<m)(f\in B_\ell\ \Leftrightarrow\ \eta(\ell)=1)\}$. By the weak
\ttt--bigness we may choose $\eta$ and $s\in\Sigt(\bar{t}\upl n)$
such that $\nor[s]>M_0$ and $\val[s]\subseteq C_\eta$. Then (by $(*)$) we
also have $\eta(\ell)=1$ and $\val[s]\subseteq B_\ell$ for all
$\ell<m$. Hence $\emptyset\neq\val[s]\subseteq\bigcap\limits_{\ell<m}
B_\ell$. 

\noindent (3)\qquad Let $\bfun:\omega\longrightarrow\omega$ witness the weat
\ttt--additivity of $(K,\Sigma)$ for $\bar{t}$. Suppose that
$\bar{D}^1,\bar{D}^2\in\sufs$, $\bar{D}=\bar{D}^1\circledast\bar{D}^2$. We
have to show that for each $n<\omega$, $\cB^n_{\bar{t}}\subseteq D_n$
(remember \ref{correctop}(1)). To this end assume that $B\in
\cB^n_{\bar{t}}$ and let $M$ be such that 
\[(\forall s\in\Sigt(\bar{t}\upl n))(\nor[s]\geq M\ \Rightarrow\ \val[s]\cap
B\neq \emptyset).\]

\begin{claim}
\label{cl5}
If $s\in \Sigt(\bar{t}\upl n)$ is such that $\nor[s]\geq \bfun(n+M)$,\\
then $\val[s]\cap \set^n_{\bar{t}}(B,\bar{D}^1)\neq\emptyset$.
\end{claim}

\begin{proof}[Proof of the Claim]
Fix $s_0\in\Sigma(t_n,\ldots,t_{m-1})$ such that $\nor[s_0]\geq
\bfun(n+M)$. Let $A=\bigcup\{f\circledast B:f\in\val[s_0]\}$. We claim that 
\begin{enumerate}
\item[$(\vartriangle)$] $A\in\cB^m_{\bar{t}}$.
\end{enumerate}
[Why? Set $N=\bfun(m+M)+957$. Suppose $s_1\in\Sigt (\bar{t}\upl m)$ has norm
$\nor[s_1]\geq N$. By the weak \ttt--additivity and the choice of $N$ we can
find $s\in\Sigt(\bar{t}\upl n)$ such that $\nor[s]\geq M$ and $\val[s]
\subseteq \{f\cup g:f\in\val[s_0],\ g\in \val[s_1]\}$. By the choice of $M$
we have $B\cap \val[s]\neq \emptyset$, so for some $f\in \val[s_0]$ and
$g\in \val[s_1]$ we have $g\in f\circledast B$. Thus $\val[s_1]\cap
A\neq\emptyset$ and we easily conclude that $A\in\cB^m_{\bar{t}}$.]

But $\bar{D}^1\in\sufs$, so $\cB^m_{\bar{t}}\subseteq D^1_m$ and hence, for
some $f\in\val[s_0]$, we get $f\circledast B\in D^1_{x_f}$. Then $f\in
\val[s_0]\cap \set^n_{\bar{t}}(B,\bar{D}^1)$. 
\end{proof}
It follows from \ref{cl5} that $\set^n_{\bar{t}}(B,\bar{D}^1)\in
\cB^n_{\bar{t}} \subseteq D^2_n$, so $B\in D_n$ as required. 
\end{proof}

\begin{corollary}
\label{new2.8}
Assume that $(K,\Sigma)$ is a tight FFCC pair with the weak \ttt--additivity
and the weak \ttt--bigness for $\bar{t}\in\PCt(K,\Sigma)$. Then there is
$\bar{D}\in\sufs$ such that $\bar{D}\circledast\bar{D}=\bar{D}$. 
\end{corollary}

\begin{proof}
By \ref{z4}+\ref{obsforassoc}(3)+\ref{correctop}+\ref{n19}. 
\end{proof}

\begin{theorem}
\label{new2.9}
Assume that $(K,\Sigma)$ is a tight FFCC pair, $\bar{t}=\langle
t_n:n<\omega\rangle\in\PCt(K,\Sigma)$. Suppose also that 
\begin{enumerate}
\item[(a)] $\bar{D}=\langle D_n:n<\omega\rangle\in\sufs$ is such that
  $\bar{D}\circledast \bar{D}=\bar{D}$, and 
\item[(b)] $\bar{A}=\langle A_n:n<\omega\rangle$ is such that $A_n\in D_n$
  for all $n<\omega$.
\end{enumerate}
Then there is $\bar{s}=\langle s_i:i<\omega\rangle\in\PCt(K,\Sigma)$ such
that $\bar{s}\geq \bar{t}$, $m^{s_0}_\dn=m^{t_0}_\dn$ and if $i<\omega$,
$s_i\in\Sigt(\bar{t}\upl k)$, then $\post(\bar{s}\upl i)\subseteq A_k$.
\end{theorem}

\begin{proof}
Let $(K,\Sigma),\bar{t},\bar{D}$ and $\bar{A}$ be as in the
assumptions. Then, in particular, $\cB^k_{\bar{t}}\subseteq D_k$ for all
$k<\omega$. 

\begin{claim}
\label{cl6}
If $M,k<\omega$ and $B\in D_k$, then there is $s\in \Sigt(\bar{t}\upl k)$
such that $\val[s]\subseteq B$, $\nor[s]\geq M$ and $(\forall f\in
\val[s])(f\circledast B\in D_{x_f})$.
\end{claim} 

\begin{proof}[Proof of the Claim]
Since $\bar{D}=\bar{D}\circledast\bar{D}$ and $B\in D_k$, we know that
$\set^k_{\bar{t}}(B,\bar{D})\in D_k$ and thus $B\cap\set^k_{\bar{t}}(B,
\bar{D}) \in D_k$. Hence $\post(\bar{t}\upl k)\setminus (B\cap
\set^k_{\bar{t}}(B,\bar{D}))\notin \cB^k_{\bar{t}}$ and we may find $s\in
\Sigt(\bar{t}\upl k)$ such that $\nor[s]\geq M$ and $\val[s]\subseteq B\cap 
\set^k_{\bar{t}}(B,\bar{D})$. This $s$ is as required in the assertion of
the claim.
\end{proof}
Now we choose inductively $s_i,B_i,k_i$ (for $i<\omega$) such that
\begin{enumerate}
\item[(i)] $B_0=A_0$, $k_0=0$,
\item[(ii)] $B_i\in D_{k_i}$, $B_i\subseteq A_{k_i}$, $k_i<k_{i+1}<\omega$,
  $s_i \in \Sigma(t_{k_i},\ldots,t_{k_{i+1}-1})$, 
\item[(iii)] $\val[s_i]\subseteq B_i$, $\nor[s_i]\geq i+1$,
\item[(iv)] if $f\in\val[s_i]$, then $B_{i+1}\subseteq f\circledast B_i\in
  D_{k_{i+1}}$. 
\end{enumerate}
Clause (i) determines $B_0$ and $k_0$. Suppose we have already chosen $k_i$
and $B_i\in D_{k_i}$. By \ref{cl6} we may find $k_{i+1}>k_i$ and $s_i\in
\Sigma(t_{k_i},\ldots, t_{k_{i+1}-1})$ such that 
\[\nor[s_i]\geq i+1,\quad \val[s_i]\subseteq B_i\quad \mbox{and }\quad
(\forall f\in\val[s_i])(f\circledast B_i\in D_{k_{i+1}}).\]
We let $B_{i+1}=A_{k_{i+1}}\cap\bigcap \{f\circledast B_i:f\in\val[s_i]\}\in
D_{k_{i+1}}$. One easily verifies the relevant demands in (ii)--(iv) for
$s_i,B_{i+1},k_{i+1}$. 
\smallskip

After the above construction is carried out, we set $\bar{s}=\langle
s_i:i<\omega\rangle$. Plainly, $\bar{s}\in\PCt(K,\Sigma)$, $\bar{s}\geq
\bar{t}$ and $m^{s_0}_\dn=m^{t_0}_\dn$.

\begin{claim}
\label{cl7}
For each $i,k<\omega$ and $s\in\Sigma(s_i,\ldots,s_{i+k})$ we have
$\val[s]\subseteq B_i$.
\end{claim}

\begin{proof}[Proof of the Claim]
Induction on $k<\omega$. If $k=0$ then the assertion of the claim follows
from clause (iii) of the choice of $s_i$. Assume we have shown the claim for
$k$. Suppose that $s\in \Sigma(s_i,\ldots,s_{i+k},s_{i+k+1})$, $i<\omega$,
and $f\in\val[s]$. Let $f_0=f\rest [m^{s_i}_\dn,m^{s_i}_\up)\in\val[s_i]$
and $f_1=f\rest [m^{s_{i+1}}_\dn,m^{s_{i+k+1}}_\up)\in\post(\bar{s}\upl
(i+1))$ (remember \ref{g1}(4)(f$^{\rm tight}$) and \ref{obsforassoc}(1)). By
the inductive hypothesis we know that $f_1\in B_{i+1}$, so by clause (iv) of
the choice of $s_i$ we get $f_1\in f_0\circledast B_i$ and thus $f=f_0\cup
f_1\in B_i$. 
\end{proof}

It follows from \ref{cl7} that for each $i<\omega$ we have $\post(\bar{s}
\upl i)\subseteq B_i\subseteq A_{k_i}$, as required.
\end{proof}

\begin{conclusion}
\label{new2.10}
Suppose that $(K,\Sigma)$ is a tight FFCC pair with weak \ttt--bigness and
weak \ttt--additivity for $\bar{t}\in\PCt(K,\Sigma)$.
\begin{enumerate}
\item[(a)] Assume that, for each $n<\omega$, $k_n<\omega$ and
  $d_n:\post(\bar{t}\upl n)\longrightarrow k_n$. Then there is $\bar{s}=
  \langle s_i:i<\omega\rangle\in\PCt(K,\Sigma)$ such that $\bar{s}\geq
  \bar{t}$, $m^{s_0}_\dn=m^{t_0}_\dn$ and for each $n<\omega$,\\
if $n$ is such that $s_i\in\Sigt(\bar{t}\upl n)$, then $d_n\rest \post(
\bar{s}\upl i)$ is constant.
\item[(b)] Suppose also that $(K,\Sigma)$ has \ttt--multiadditivity. Let 
  $d_n:\post(\bar{t}\upl n)\longrightarrow k$ (for $n<\omega$),
  $k<\omega$. Then there are $\bar{s}=\langle s_i:i<\omega\rangle \in
  \PCt(K,\Sigma)$ and $\ell<k$ such that $\bar{s}\geq \bar{t}$ and for each
  $i<\omega$, if $n$ is such that $s_i\in\Sigt(\bar{t}\upl n)$ and
  $f\in\post(\bar{s}\upl i)$, then $d_n(f)=\ell$. 
\end{enumerate}
\end{conclusion}

Now we will use \ref{new2.10} to give a new proof of Carlson--Simpson
Theorem. This theorem was used as a crucial lemma in the (inductive) proof
of the Dual Ramsey Theorem \cite[Theorem 2.2]{CS84}.

\begin{theorem}
\label{CS6.3new}
[Carlson and Simpson {\cite[Theorem 6.3]{CS84}}]
Suppose that $0<N<\omega$, $\bbX=\bigcup\limits_{n<\omega} {}^n N$ and
$\bbX=C_0\cup\ldots\cup C_k$, $k<\omega$. Then there exist a partition
$\{Y\}\cup \{Y_i:i<\omega\}$ of $\omega$ and a function $f:Y\longrightarrow
N$ such that 
\begin{enumerate}
\item[(a)] each $Y_i$ is a finite non-empty set,
\item[(b)] if $i<j<\omega$ then $\max(Y_i)<\min(Y_j)$,
\item[(c)] for some $\ell\leq k$:\\
if $i<\omega$, $g:\min(Y_i)\longrightarrow N$, $f\rest \min(Y_i)\subseteq g$
and $g\rest Y_j$ is constant for $j<i$, then $g\in C_\ell$.
\end{enumerate}
\end{theorem}

\begin{proof}
For $f\in\bbX$ let $d_0(f)=\min\{\ell\leq k:f\in C_\ell\}$. Consider the
tight FFCC pair $(K_N,\Sigma_N)$ defined in Example \ref{CarlSim}. It
satisfies the assumptions of \ref{new2.10}. Fix any $\bar{t}\in\PCt(K_N,
\Sigma_N)$ with $m^{t_0}_\dn=0$ and use \ref{new2.10}(a) to choose
$\bar{s}\in\PCt(K_N,\Sigma_N)$ such that $\bar{s}\geq\bar{t}$, $m^{s_0}_\dn=
m^{t_0}_\dn=0$ and $d_0\rest \post(\bar{s})$ is constant. Set
$Y=\bigcup\limits_{i<\omega} X_{s_i}$, $f=\bigcup\limits_{i<\omega}
\varphi_{s_i}$ and $Y_i=[m^{s_i}_\dn,m^{s_i}_\up)\setminus X_{s_i}$ for
$i<\omega$. 
\end{proof}

\section{Very weak bigness}
The assumptions of Conclusion \ref{new2.10} (weak \ttt--bigness and weak
\ttt--additivity) are somewhat strong. We will weaken them substantially
here, getting weaker but still interesting conclusion.  

\begin{definition}
\label{weaakbig}
Let $(K,\Sigma)$ be a tight FFCC pair for $\bH$, $\bar{t}\in
\PCt(K,\Sigma)$. 
\begin{enumerate}
\item For $n<m<\omega$ we define 
\[\pos(\bar{t}\rest [n,m))=\bigcup\{\val[s]:s\in\Sigma(t_n,\ldots,
t_{m-1})\}\]
and we also keep the convention that $\pos(\bar{t}\rest
[n,n))=\{\emptyset\}$.\\ 
\relax [Note that $\pos(\bar{t}\rest [n,m))=\{f_n\cup\ldots\cup f_{m-1}:
f_\ell\in\val[t_\ell]\mbox{ for }\ell<m\}$ (remember \ref{g1}(4)(f$^{\rm
  tight}$) and \ref{obsforassoc}(1)).] 
\item We say that $(K,\Sigma)$ has {\em the very weak \ttt--bigness for
    $\bar{t}$} if 
\begin{enumerate}
\item[$(\boxtimes)^{\rm vw}_{\bar{t}}$] for every $n,L,M<\omega$ and a
  partition $\cF_0\cup\ldots\cup\cF_L=\post(\bar{t}\upl n)$, there are
  $i_0=n\leq i_1<i_2\leq i_3$, $\ell\leq L$ and $g_0\in\pos(\bar{t}\rest
  [i_0,i_1))$, $g_2\in\pos(\bar{t}\rest [i_2,i_3))$ and $s\in
  \Sigma(t_{i_1},\ldots,t_{i_2-1})$ such that 
\[\nor[s]\geq M\quad\mbox{ and }\quad (\forall g_1\in\val[s])( g_0\cup
g_1\cup g_2\in\cF_\ell).\]
\end{enumerate}
\end{enumerate}
\end{definition}

\begin{observation}
\label{obsweak}
If a tight FFCC pair $(K,\Sigma)$ has the weak \ttt--bigness for $\bar{t}$,
then it has the very weak \ttt--bigness for $\bar{t}$. 
\end{observation}

\begin{definition}
\label{newsufstar}
Let $(K,\Sigma)$ be a tight FFCC pair for $\bH$ and $\bar{t}\in\PCt(K,
\Sigma)$. 
\begin{enumerate}
\item For $n<\omega$, $\cC^n_{\bar{t}}$ is the family of all sets
  $B\subseteq \post(\bar{t}\upl n)$ such that for some $M$ we have:\\
if $i_0=n\leq i_1<i_2\leq i_3$, $g_0\in\pos(\bar{t}\rest [i_0,i_1))$,
$g_2\in \pos(\bar{t}\rest [i_2,i_3))$ and $s\in\Sigma(t_{i_1},\ldots,
t_{i_2-1})$, $\nor[s]\geq M$, then $B\cap \{g_0\cup g_1\cup g_2:g_1\in
\val[s]\}\neq\emptyset$.  
\item $\sufp$ is the family of all $\bar{D}=\langle D_n:n<\omega\rangle \in
  \suft$ such that $\cC^n_{\bar{t}}\subseteq D_n$ for all $n<\omega$. 
\end{enumerate}
\end{definition}

\begin{proposition}
\label{newn19}
Let $(K,\Sigma)$ be a tight FFCC pair for $\bH$ and
$\bar{t}\in\PCt(K,\Sigma)$.
\begin{enumerate}
\item $\sufp$ is a closed subset of $\suft$, $\sufs\subseteq\sufp$.
\item If $(K,\Sigma)$ has the very weak \ttt--bigness for $\bar{t}$, then
  $\sufp\neq\emptyset$. 
\item If $\bar{D}\in\sufp$, $n<\omega$ and $B\in\cC^n_{\bar{t}}$, then
  $\set^n_{\bar{t}}(B,\bar{D})=\post(\bar{t}\upl n)$. 
\item $\sufp$ is closed under the operation $\circledast$ (defined in
  \ref{sufs}(6)).  
\end{enumerate}
\end{proposition}

\begin{proof}
(1)\quad Since in \ref{newsufstar}(1) we allow $i_1=i_0$ and $i_3=i_2$
(so $g_0=g_2=\emptyset$), we easily see that $\cC^n_{\bar{t}}\subseteq
\cB^n_{\bar{t}}$. Hence $\sufs\subseteq\sufp$. The proof that $\sufp$ is
closed is the same as for \ref{n19}(1).

\noindent (2)\quad Like \ref{n19}(2).

\noindent (3)\quad Let $M$ be such that $B\cap \{g_0\cup g_1\cup g_2: g_1\in
\val[s]\}\neq\emptyset$ whenever $g_0\in\pos(\bar{t}\rest [n,i_1))$,
$g_2\in\pos(\bar{t}\rest [i_2,i_3))$, $s\in\Sigma(\bar{t}\rest [i_1,i_2))$,
$\nor[s]\geq M$, $n\leq i_1<i_2\leq i_3$. We will show that this $M$
witnesses $f\circledast B\in \cC^{x_f}_{\bar{t}}$ for all $f\in\post(t\upl
n)$. 

So suppose that $f\in\post(\bar{t}\upl n)$ and $x_f\leq i_1<i_2\leq i_3$,
$g_0\in\pos(\bar{t}\rest [x_f,i_1))$, $s\in\Sigma(\bar{t}\rest [i_1,i_2))$,
$\nor[s]\geq M$ and $g_2\in\pos(\bar{t}\rest [i_2,i_3))$. Then $f\cup
g_0\in\pos(\bar{t}\rest [n,i_1))$ (remember \ref{obsforassoc}(1)) and
consequently (by the choice of $M$) $B\cap\big\{(f\cup g_0)\cup g_1\cup
g_2:g_1\in\val[s]\big\}\neq \emptyset$. Let $g^*_1\in\val[s]$ be such that
$f\cup g_0\cup g^*_1\cup g_2\in B$. Then $g_0\cup g^*_1\cup g_2\in
f\circledast B$ witnessing that $(f\circledast B)\cap\big\{g_0\cup g_1\cup
g_2:g_1\in\val[s]\big\}\neq \emptyset$.  

Since $\cC^{x_f}_{\bar{t}}\subseteq D_{x_f}$ we conclude now that
$f\circledast B\in D_{x_f}$ so $f\in\set^n_{\bar{t}}(B,\bar{D})$. 

\noindent (4)\quad Suppose $\bar{D}_1,\bar{D}_2\in\sufp$, $\bar{D}=\bar{D}_1
\circledast \bar{D}_2$. Let $B\in\cC^n_{\bar{t}}$, $n<\omega$. By (3) we
know that $\set^n_{\bar{t}}(B,\bar{D}_1)=\post(\bar{t}\upl n)\in D^2_n$ and
thus $B\in D_n$. Consequently, $\cC^n_{\bar{t}}\subseteq D_n$ for all
$n<\omega$, so $\bar{D}\in\sufp$. 
\end{proof}

\begin{corollary}
\label{fancy2.8}
Assume that $(K,\Sigma)$ is a tight FFCC pair with the very weak
\ttt--bigness for $\bar{t}\in\PCt(K,\Sigma)$. Then there is
$\bar{D}\in\sufp$ such that $\bar{D}\circledast\bar{D}=\bar{D}$.
\end{corollary}

\begin{theorem}
\label{fancy2.9}
Assume that $(K,\Sigma)$ is a tight FFCC pair for $\bH$, $\bar{t}=\langle
t_n:n<\omega\rangle\in\PCt(K,\Sigma)$. Let $\bar{D}\in\sufp$ be such that
$\bar{D}\circledast\bar{D}=\bar{D}$ and suppose that $A_n\in D_n$ for
$n<\omega$. Then there are sequences $\langle n_i:i<\omega\rangle$, $\langle
g_{3i},g_{3i+2}:i<\omega\rangle$ and $\langle s_{3i+1}:i<\omega\rangle$ such
that for every $i<\omega$:
\begin{enumerate}
\item[$(\alpha)$] $0=n_0\leq n_{3i}\leq n_{3i+1}<n_{3i+2}\leq
  n_{3i+3}<\omega$, 
\item[$(\beta)$] if $j=3i$ or $j=3i+2$, then $g_j\in\pos(\bar{t}\rest
  [n_j,\ldots,n_{j+1}))$, 
\item[$(\gamma)$] if $j=3i+1$, then $s_j\in\Sigma(t_{n_j}, \ldots,
  t_{n_{j+1}-1})$ and $\nor[s_j]\geq j$,  
\item[$(\delta)$] if $g_{3\ell+1}\in\val[s_{3\ell+1}]$ for $\ell\in [i,k)$,
  $i<k$, then $\bigcup\limits_{j=3i}^{3k-1} g_j\in A_{n_{3i}}$.
\end{enumerate}
\end{theorem}

\begin{proof}
Parallel to \ref{new2.9}, just instead of $\val[s_i]$ use
$\{g_{i-1}\cup g\cup g_{i+1}:g\in\val[s_i]\}$. 
\end{proof}

\begin{conclusion}
\label{fancy2.10}
Assume that $(K,\Sigma)$ is a tight FFCC pair with the very weak
\ttt--bigness for $\bar{t}\in\PCt(K,\Sigma)$. Suppose that for each
$n<\omega$ we are given $k_n<\omega$ and a mapping $d_n:\post(\bar{t}\upl
n)\longrightarrow k_n$. Then there are sequences $\langle n_i:i<\omega
\rangle$, $\langle g_{3i},g_{3i+2}: i<\omega\rangle$, $\langle s_{3i+1}: i<
\omega\rangle$ and $\langle c_i:i<\omega\rangle$ such that for each
$i<\omega$:  
\begin{enumerate}
\item[$(\alpha)$] $0=n_0\leq n_{3i}\leq n_{3i+1}<n_{3i+2}\leq
  n_{3i+3}<\omega$, $c_i\in k_{n_{3i}}$, 
\item[$(\beta)$] if $j=3i$ or $j=3i+2$, then $g_j\in\pos(\bar{t}\rest
  [n_j,\ldots,n_{j+1}))$, 
\item[$(\gamma)$] if $j=3i+1$, then $s_j\in\Sigma(t_{n_j}, \ldots,
  t_{n_{j+1}-1})$ and $\nor[s_j]\geq j$,  
\item[$(\delta)$] if $i<k$ and $f\in\pos(\bar{t}\rest [n_{3i},n_{3k}))$ are 
  such that 
\[g_{3\ell}\cup g_{3\ell+2}\subseteq f\quad\mbox{ and }\quad f\rest
[m^{s_{3\ell+1}}_\dn, m^{s_{3\ell+1}}_\up)\in\val[s_{3\ell+1}]\quad \mbox{
  for all }\ell\in [i,k),\]
then $d_{n_{3i}}(f)=c_i$.
\end{enumerate}
\end{conclusion}

\begin{example}
\label{re388}
Let $(G,\circ)$ be a finite group. For a function $f:S\longrightarrow G$ and
$a\in G$ we define $a\circ f:S\longrightarrow G$ by $(a\circ f)(x)= a\circ
f(x)$ for $x\in S$. Let $\bH_G(m)=G$ (for $m<\omega$) and let $K_G$ consist
of all FP creatures $t$ for $\bH_G$ such that 
\begin{itemize}
\item $\nor[t]=m^t_\up$, $\dis[t]=\emptyset$, 
\item $\val[t]\subseteq {}^{[m^t_\dn,m^t_\up)}G$ is such that $(\forall
  f\in\val[t])(\forall a\in G)(a\circ f\in\val[t])$. 
\end{itemize}
For $t_0,\ldots,t_n\in K_G$ with $m^{t_{\ell+1}}_\dn=m^{t_\ell}_\up$ (for
$\ell<n$) we let $\Sigma_G(t_0,\ldots,t_n)$ consist of all creatures $t\in
K_G$ such that
\begin{itemize}
\item $m^t_\dn=m^{t_0}_\dn$, $m^t_\up=m^{t_n}_\up$,
\item $\val[t]\subseteq \{f\in {}^{[m^t_\dn,m^t_\up)}G: (\forall \ell\leq
  n)(f\rest [m^{t_\ell}_\dn,m^{t_\ell}_\up)\in\val[t_\ell])\}$. 
\end{itemize}
Then
\begin{enumerate}
\item $(K_G,\Sigma_G)$ is a tight FFCC pair for $H_G$.
\item If $|G|=2$, then $(K_G,\Sigma_G)$ has the very weak \ttt--bigness for
  every candidate $\bar{t}\in\PCt(K_G,\Sigma_G)$. 
\end{enumerate}
\end{example}

\begin{proof}
(1)\quad Straightforward.

\noindent (2)\quad Let $G=(\{-1,1\},\cdot)$. Suppose that
$\bar{t}\in\PCt(K_G,\Sigma_G)$ and $\post(\bar{t}\upl n)=\cF_0\cup \ldots
\cup \cF_L$, $n,L,M<\omega$. For future use we will show slightly more than
needed for the very weak bigness.

We say that $N\geq n+M$ is $\ell$--good (for $\ell\leq L$) if 
\begin{enumerate}
\item[$(\boxdot)_\ell$] there are $j_2\geq j_1>N$, $g_0\in\pos(\bar{t}\rest
  [n,N))$, $g_2\in\pos(\bar{t}\rest [j_1,j_2))$ and $s\in\Sigma_G(\bar{t}
  \rest [N,j_1))$ such that $\{g_0\cup g_1\cup g_2:g_1\in\val[s]\}\subseteq
  \cF_\ell$. 
\end{enumerate}
(Note that if $s$ is as in $(\boxdot)_\ell$, then also
$\nor[s]=m^{t_{j_1}}_\dn \geq j_1>N\geq M$.) We are going to argue that 
\begin{enumerate}
\item[$(\odot)$] almost every $N\geq n+M$ is $\ell$--good for some $\ell\leq
  L$.  
\end{enumerate}
So suppose that $(\odot)$ fails and we have an increasing sequence
$n+M<N(0)<N(1)<N(2)<\ldots$ such that $N(k)$ is not $\ell$--good for any
$\ell\leq L$ (for all $k<\omega$). Let $m=L+957$ and for each $i\in
[n,N(m)]$ fix $f_i\in\val[t_i]$ (note that then $-f_i\in\val[t_i]$ as
well). Next, for $j<m$ define
\[h_j=\bigcup_{i=n}^{N(j)-1} f_i\cup\bigcup_{i=N(j)}^{N(m)} -f_i\]
and note that $h_j\in\post(\bar{t}\upl n)$. For some $\ell\leq L$ and
$j<k<m$ we have $h_j,h_k\in\cF_\ell$. Set 
\[\begin{array}{l}
\displaystyle g_0=\bigcup_{i=n}^{N(j)-1}f_i=h_j\rest m^{t_{N(j)}}_\dn 
=h_k\rest m^{t_{N(j)}}_\dn,\\
\displaystyle g_2=\bigcup_{i=N(k)}^{N(m)} -f_i=h_j\rest [m^{t_{N(k)}}_\dn,
m^{t_{N(m)}}_\up) =h_k\rest [m^{t_{N(k)}}_\dn,m^{t_{N(m)}}_\up)
\end{array}\]
and let $s\in\Sigma_G(\bar{t}\rest [N(j),N(k)))$ be such that 
\[\val[s]=\{h_j\rest [m^{t_{N(j)}}_\dn,m^{t_{N(k)}}_\dn), h_k\rest [
m^{t_{N(j)}}_\dn, m^{t_{N(k)}}_\dn)\}.\]
Then $\{g_0\cup g_1\cup g_2: g_1\in\val[s]\}= \{h_j,h_k\}\subseteq
\cF_\ell$, so $g_0,g_2$ and $s$ witness $(\boxdot)_\ell$ for $N(j)$, a
contradiction.   
\end{proof}

The following conclusion is a special case of the partition theorem used in
Goldstern and Shelah \cite{GoSh:388} to show that a certain forcing notion
preserves a Ramsey ultrafilter (see \cite[3.9, 4.1 and Section
5]{GoSh:388}). 

\begin{corollary}
\label{case388}
Let $\bbY=\bigcup\limits_{n<\omega}{}^n\{-1,1\}$. Suppose that $\bbY=C_0
\cup\ldots\cup C_L$, $L<\omega$. Then there are a sequence $\langle
n_i:i<\omega\rangle$, a function $f:\omega\longrightarrow \{-1,1\}$ and
$\ell<L$ such that 
\begin{enumerate}
\item[(a)] $0=n_0\leq n_{3i}\leq n_{3i+1}<n_{3i+2}\leq n_{3i+3}<\omega$,
\item[(b)] if $g:n_{3i}\longrightarrow \{-1,1\}$ for each $j<i$ satisfies 
\[\begin{array}{l}
g\rest [n_{3j},n_{3j+1})\cup g\rest [n_{3j+2},n_{3j+3})\subseteq f\quad
\mbox{ and}\\
g\rest [n_{3j+1},n_{3j+2})\in \{f\rest [n_{3j+1},n_{3j+2}), -f \rest
[n_{3j+1},n_{3j+2})\}
\end{array}\]  
then $g\in C_\ell$.
\end{enumerate}
\end{corollary}

\begin{proof}
  By \ref{fancy2.10}+\ref{re388}.
\end{proof}

\section{Limsup candidates}

\begin{definition}
\label{limsupdef}
Let $(K,\Sigma)$ be a tight FFCC pair for $\bH$ and $J$ be an ideal on
$\omega$. 
\begin{enumerate}
\item {\em A ${\rm limsup}_J$--candidate\/} for $(K,\Sigma)$ is a sequence 
  $\bar{t}=\langle t_n:n<\omega\rangle$ such that $t_n\in K$, $m^{t_n}_\up=
  m^{t_{n+1}}_\dn$ (for all $n$) and for each $M$
\[\{m^{t_n}_\dn:n<\omega\ \&\ \nor[t_n]>M\}\in J^+.\]
The family of all ${\rm limsup}_J$--candidates for $(K,\Sigma)$ is denoted
by $\PCJ$. 
\item {\em A finite candidate\/} for $(K,\Sigma)$ is a finite sequence
  $\bar{s}=\langle s_n:n<N\rangle$, $N<\omega$, such that $s_n\in K$ and 
  $m^{s_n}_\up= m^{s_{n+1}}_\dn$ (for $n<N$). The family of all finite
  candidates is called $\FC$. 
\item For $\bar{s}=\langle s_n:n<N\rangle\in\FC$ and $M<\omega$ we set 
\[\base_M(\bar{s})=\{m^{s_n}_\dn:n<N\ \&\ \nor[s_n]\geq M\}.\]
\item Let $\bar{t},\bar{t}'\in\PCJ$, $\bar{s}\in\FC$. Then we define
  $\bar{t}\upl n$, $\Sigt(\bar{t})$, $\post(\bar{t})$, $\bar{t}\leq
  \bar{t}'$, $\pos(\bar{t}\rest [n,m))$ and $\pos(\bar{s})$ as in the case
  of tight pure candidates (cf.~\ref{g3}, \ref{weaakbig}).
\item Let $\bar{t}\in\PCJ$. The family of all finite candidates $\bar{s}=
  \langle s_n:n<N\rangle\in\FC$ satisfying
\[(\forall n<N)(\exists k,\ell)(s_n\in\Sigma(\bar{t}\rest [k,\ell)))\quad
\mbox{ and }\quad m^{s_0}_\dn=m^{t_0}_\dn\]
is denoted by $\Sigseq(\bar{t})$. 
\end{enumerate}
\end{definition}

\begin{definition}
\label{Jbig}
Let $(K,\Sigma)$ be a tight FFCC pair, $J$ be an ideal on $\omega$ and
$\bar{t}\in \PCJ$.
\begin{enumerate}
\item We say that $(K,\Sigma)$ {\em has the $J$--bigness for $\bar{t}$\/} if 
  \begin{enumerate}
  \item[$(\otimes)^J_{\bar{t}}$] for every $n,L,M<\omega$ and a partition
    $\cF_0\cup\ldots\cup\cF_L=\post(\bar{t}\upl n)$, there are $\ell\leq L$
    and a set $Z\in J^+$ such that  
\[(\forall z\in Z)(\exists \bar{s}\in\Sigseq(\bar{t}\upl n))(z\in
\base_M(\bar{s})\ \&\ \pos(\bar{s})\subseteq \cF_\ell).\]
  \end{enumerate}
\item The pair $(K,\Sigma)$ {\em captures singletons\/}
  (cf. \cite[2.1.10]{RoSh:470}) if  
\[(\forall t\in K)(\forall f\in \val[t])(\exists s\in\Sigma(t))(\val[s]=
\{f\}).\] 
\item We define $\suft$ as in \ref{sufs}(3), $\set^n_{\bar{t}}(A,\bar{D})$
  (for $A\subseteq \post(\bar{t}\upl n)$ and $\bar{D}\in \suft$)
  as in \ref{sufs}(5) and the operation $\circledast$ on $\suft$ as in
  \ref{sufs}(6).
\item For $n<\omega$, $\cD^{n,J}_{\bar{t}}$ is the family of all sets
  $B\subseteq\post(\bar{t}\upl n)$ such that for some $M<\omega$ and $Y\in
  J^c$ we have:

if $\bar{s}\in\Sigseq(\bar{t}\upl n)$ and $\base_M(\bar{s})\cap Y\neq
\emptyset$, then $B\cap\pos(\bar{s})\neq \emptyset$. 
\item $\sufJ$ is the family of all $\bar{D}=\langle D_n:n<\omega\rangle
  \in\suft$ such that $\cD^{n,J}_{\bar{t}}\subseteq D_n$ for all $n<\omega$.  
\end{enumerate}
\end{definition}

\begin{remark}
Note that no norms were used in the proofs of \ref{obsforassoc},
\ref{correctop}, so those statements are valid for the case of $\bar{t}\in
\PCJ$ too.   
\end{remark}

\begin{observation}
\label{getJbig}
\begin{enumerate}
\item Assume that $(K,\Sigma)$ is a tight FFCC pair with bigness (see
  \ref{defbig}(3)). If $(K,\Sigma)$ captures singletons or it has the
  $\ttt$--multiadditivity (see \ref{defadd}(3)), then $(K,\Sigma)$ has the
  $J$--bigness for any $\bar{t}\in \PCJ$.  
\item The tight FFCC pairs $(K_1,\Sigma_1^*)$, $(K_3,\Sigma_3^*)$ and
  $(K_N,\Sigma_N)$ defined in \ref{ex1.10}, \ref{ex1.12} and \ref{CarlSim},
  respectively, have $J$--bigness on every $\bar{t}\in\PCJ$.
\end{enumerate}
\end{observation}

Every tight FFCC pair can be extended to a pair capturing singletons
while preserving $\post(\bar{t})$.

\begin{definition}
Let $(K,\Sigma)$ be a tight FFCC pair for $\bH$. Define $K^{\rm sin}$ as the
family of all FP creatures $t$ for $\bH$ such that 
\[\dis[t]=K,\quad \nor[t]=0\quad\mbox{ and }\quad |\val[t]|=1.\]
Then we let $K^s=K\cup K^{\rm sin}$ and for $t_0,\ldots,t_n\in K^s$ with
$m^{t_\ell}_\up= m^{t_{\ell+1}}_\dn$ (for $\ell<n$) we set
\begin{itemize}
\item $\Sigma^{\rm sin}(t_0,\ldots,t_n)$ consists of all creatures $t\in
  K^{\rm sin}$ such that $m^t_\dn=m^{t_0}_\dn$, $m^t_\up=m^{t_n}_\up$ and 

if $\val[t]=\{f\}$ then $f\rest [m^{t_\ell}_\dn,m^{t_\ell}_\up)\in
\val[t_\ell]$ for all $\ell\leq n$;  
\item if $t_0,\ldots,t_n\in K$, then $\Sigma^s(t_0,\ldots,t_n)=
  \Sigma(t_0,\ldots,t_n)\cup \Sigma^{\rm sin}(t_0,\ldots,t_n)$;
\item if $t_\ell\in K^{\rm sin}$ for some $\ell\leq n$, then
  $\Sigma^s(t_0,\ldots,t_n)= \Sigma^{\rm sin}(t_0,\ldots,t_n)$. 
\end{itemize}
\end{definition}

\begin{observation}
Let $(K,\Sigma)$ be a tight FFCC pair for $\bH$.
\begin{enumerate}
\item $(K^s,\Sigma^s)$ is a tight FFCC pair for $\bH$ and it captures
  singletons. 
\item If $(K,\Sigma)$ has bigness then so does $(K^s,\Sigma^s)$ and
  consequently then $(K^s,\Sigma^s)$ has the $J$--bigness on any $\bar{t}\in
  {\rm PC}^J_{{\rm w}\infty}(K^s,\Sigma^s)$. 
\item If $\bar{t}\in \PCJ$, then $\bar{t}\in {\rm PC}^J_{{\rm
      w}\infty}(K^s, \Sigma^s)$ and $\post(\bar{t})$ with respect to
  $(K,\Sigma)$ is the same as $\post(\bar{t})$ with respect to
  $(K^s,\Sigma^s)$.  
\end{enumerate}
\end{observation}

\begin{observation}
  Let $G=(\{-1,1\},\cdot)$ and $(K_G,\Sigma_G)$ be the tight FFCC pair
  defined in \ref{re388}. Suppose that $\bar{t}\in   {\rm PC}^J_{{\rm
      w}\infty} (K_G,\Sigma_G)$. Then $(K^s_G,\Sigma^s_G)$ (sic!) has the
  $J$--bigness for $\bar{t}$.
\end{observation}

\begin{proof}
  Note that ${\rm PC}^J_{{\rm w}\infty}(K_G,\Sigma_G)\subseteq
  \PCt(K_G,\Sigma_G)$ and remember $(\odot)$ from the proof of
  \ref{re388}(2).  
\end{proof}

\begin{proposition}
Assume that $(K,\Sigma)$ is a tight FFCC pair for $\bH$, $J$ is an ideal on
$\omega$ and $\bar{t}\in\PCJ$.  
\begin{enumerate}
\item $\sufJ$ is a closed subset of the compact Hausdorff topological space
  $\suft$. 
\item If $(K,\Sigma)$ has the $J$--bigness for $\bar{t}$, then $\sufJ\neq
  \emptyset$. 
\item If $\bar{D}\in\sufJ$, $n<\omega$ and $B\in\cD^{n,J}_{\bar{t}}$, then
  $\set^n_{\bar{t}}(B,\bar{D})\in\cD^{n,J}_{\bar{t}}$. 
\item $\sufJ$ is closed under the operation $\circledast$. 
\end{enumerate}
\end{proposition}

\begin{proof}
(1)\quad Same as \ref{n19}(1).

\noindent (2)\quad  Similar to \ref{n19}(2).

\noindent (3)\quad Let $B\in\cD^{n,J}_{\bar{t}}$ be witnessed by $M<\omega$
and $Z\in J^c$. We are going to show that then for each
$\bar{s}\in\Sigseq(\bar{t}\upl n)$ with $\base_M(\bar{s})\cap
Z\neq\emptyset$ we have $\pos(\bar{s})\cap \set^n_{\bar{t}}(B,\bar{D}) \neq
\emptyset$. So let $\bar{s}=\langle s_0,\ldots,s_k\rangle\in
\Sigseq(\bar{t}\upl n)$, $\base_M(\bar{s})\cap Z\neq\emptyset$ and let $x$
be such that $m^{s_k}_\up=m^{t_x}_\dn$. Set $A=\bigcup\{f\circledast B:f\in
\pos(\bar{s})\}$. Suppose that $\bar{r}\in\Sigseq(\bar{t}\upl x)$. Then
$\bar{s}\conc\bar{r}\in\Sigseq(\bar{t}\upl n)$ and $\base_M(\bar{s}\conc
\bar{r})\supseteq \base_M(\bar{s})$, so $\pos(\bar{s}\conc\bar{r})\cap B\neq
\emptyset$. Let $g\in\pos(\bar{s}\conc\bar{r})\cap B$ and $f_0=g\rest
m^{s_k}_\up$, $f_1=g\rest [m^{s_k}_\up,\omega)$. Necessarily
$f_0\in\pos(\bar{s})$, $f_1\in\pos(\bar{r})$ and (as $g=f_0\cup f_1\in B$)
$f_1\in f_0\circledast B$. Consequently $A\cap \pos(\bar{r})\neq
\emptyset$. Now we easily conclude that $A\in \cD^{x,J}_{\bar{t}}\subseteq
D_x$. Hence for some $f\in \pos(\bar{s})$ we have $f\circledast B\in D_x$,
so $f\in\set^n_{\bar{t}}(B,\bar{D})$.

\noindent (4)\quad Follows from (3). 
\end{proof}

\begin{corollary}
  Assume that $(K,\Sigma)$ is a tight FFCC pair, $J$ is an ideal on $\omega$
  and $\bar{t}\in\PCJ$. If $(K,\Sigma)$ has the $J$--bigness for $\bar{t}$,
  then there is $\bar{D}\in\sufJ$ such that $\bar{D}\circledast
  \bar{D}=\bar{D}$. 
\end{corollary}

\begin{definition}
\label{game}
Let $J$ be an ideal on $\omega$.
\begin{enumerate}
\item A game $\Game_J$ between two players, One and Two, is defined as
  follows. A play of $\Game_J$ lasts $\omega$ steps in which the players
  construct a sequence $\langle Z_i,k_i:i<\omega\rangle$. At a stage $i$ of
  the play, first One chooses a set $Z_i\in J^+$ and then Two answers with
  $k_i\in Z_i$. At the end, Two wins the play $\langle Z_i,k_i:i<\omega
  \rangle$ if and only if $\{k_i:i<\omega\}\in J^+$. 
\item We say that $J$ is {\em an R--ideal\/} if player One has no winning
  strategy in $\Game_J$.
\end{enumerate}
\end{definition}

\begin{remark}
  If $J$ is a maximal ideal on $\omega$, then it is an R--ideal if and only
  if the dual filter $J^c$ is a Ramsey ultrafilter. Also, the ideal
  $[\omega]^{<\omega}$ of all finite subsets of $\omega$ is an R--ideal.
\end{remark}

\begin{theorem}
Assume that $(K,\Sigma)$ is a tight FFCC pair, $J$ is an R--ideal on
$\omega$ and $\bar{t}\in\PCJ$. Suppose that $\bar{D}\in\sufJ$ satisfies
$\bar{D}\circledast\bar{D}=\bar{D}$ and let $A_n\in D_n$ for
$n<\omega$. Then there are $\bar{s}\in\PCJ$ and $0=k(0)<k(1)< k(2)<k(3)
<\ldots <\omega$ such that $\bar{t}\leq \bar{s}$, $m^{s_0}_\dn=m^{t_0}_\dn$
and

if $i<j$, $\ell<\omega$, $s_{k(i)}\in\Sigt(\bar{t}\upl \ell)$, then
$\pos(s_{k(i)},s_{k(i)+1},\ldots,s_{k(j)-1})\subseteq A_\ell$. 
\end{theorem}

\begin{proof}
The proof follows the pattern of \ref{new2.9} with the only addition that we
need to make sure that at the end $\bar{s}\in\PCJ$, so we play a round of
$\Game_J$. First,

\begin{claim}
\label{cl8}
Assume $M,\ell<\omega$ and $B\in D_\ell$. Then for some set $Z\in J^+$, for
every $x\in Z$, there is $\bar{s}\in\Sigseq(\bar{t}\upl \ell)$ such that 
\[x\in\base_M(\bar{s}),\quad \pos(\bar{s})\subseteq B\quad\mbox{ and }\quad
(\forall f\in \pos(\bar{s}))(f\circledast B\in D_{x_f}).\]
\end{claim}

\begin{proof}[Proof of the Claim]
Similar to \ref{cl6}. Since $\bar{D}\circledast\bar{D}=\bar{D}$ and $B\in
D_\ell$, we know that $\post(\bar{t}\upl \ell)\setminus (B\cap
\set^\ell_{\bar{t}}(B,\bar{D}))\notin \cD^{\ell,J}_{\bar{t}}$. Therefore, for
each $Y\in J^c$ there are $x\in Y$ and $\bar{s}\in\Sigseq(\bar{t}\upl \ell)$
such that $x\in\base_M(\bar{s})$ and $\pos(\bar{s})\subseteq
B\cap\set^\ell_{\bar{t}}(B,\bar{D})$. So the set $Z$ of $x$ as above belongs
to $J^+$.    
\end{proof}

Consider the following strategy for player One in the game $\Game_J$. During
the course of a play, in addition to his innings $Z_i$, One chooses aside
$\ell_i<\omega$, $B_i\in D_{\ell_i}$ and $\bar{s}^i\in\FC$. So suppose that
the players have arrived to a stage $i$ of the play and a sequence $\langle
Z_j,k_j,\bar{s}^j,\ell_j,B_j:j<i\rangle$ has been constructed. Stipulating
$\ell_{-1}=0$ and $B_{-1}=A_0$, One uses \ref{cl8} to pick a set 
$Z_i\subseteq\omega\setminus m^{t_{\ell_{i-1}}}_\dn$ such that $Z_i\in J^+$
and for all $x\in Z_i$ there exists $\bar{s}\in\Sigseq(\bar{t}\upl
\ell_{i-1})$ with  
\[x\in\base_{i+1}(\bar{s})\ \ \&\ \ \pos(\bar{s})\subseteq B_{i-1}\ \ \&\ \
(\forall f\in \pos(\bar{s}))(f\circledast B_{i-1}\in D_{x_f}).\] 
The set $Z_i$ is One's inning in $\Game_J$ after which Two picks $k_i\in
Z_i$. Now, One chooses $\bar{s}^i\in\Sigseq(\bar{t}\upl \ell_{i-1})$ such
that 
\begin{enumerate}
\item[$(\alpha)_i$] $k_i\in\base_{i+1}(\bar{s}^i)$,
\item[$(\beta)_i$]  $\pos(\bar{s}^i)\subseteq B_{i-1}$, and
\item[$(\gamma)_i$] $(\forall f\in\pos(\bar{s}^i))(f\circledast B_{i-1} \in
  D_{x_f})$.   
\end{enumerate}
He also sets
\begin{enumerate}
\item[$(\delta)_i$] $\ell_i=x_f$ for all (equivalently: some)
  $f\in\pos(\bar{s}^i)$, and
\item[$(\varepsilon)_i$] $B_i=A_{\ell_i}\cap\bigcap\{f\circledast
  B_{i-1}:f\in \pos(\bar{s}^i)\}\in D_{\ell_i}$.
\end{enumerate}

The strategy described above cannot be winning for One, so there is a play
$\langle Z_i,k_i:i<\omega\rangle$ in which One follows the strategy, but
$\{k_i:i<\omega\}\in J^+$. In the course of this play One constructed aside
a sequence $\langle \ell_i,B_i,\bar{s}^i:i<\omega\rangle$ such that
$\bar{s}^i \in\Sigseq(\bar{t}\upl\ell_{i-1})$ and conditions
$(\alpha)_i$---$(\varepsilon)_i$ hold (where we stipulate $\ell_{-1}=0$,
$B_{-1}=A_0$). Note that $\bar{s}^i\conc\bar{s}^{i+1}\conc\ldots\conc
\bar{s}^{i+k}\in \Sigseq(\bar{t}\upl \ell_{i-1})$ for each
$i,k<\omega$. Also 
\[\bar{s}\stackrel{\rm def}{=}\bar{s}^0\conc\bar{s}^{1}\conc\bar{s}^{2}
\conc\ldots\in\PCJ\quad \mbox{ and }\quad \bar{s}\geq \bar{t}.\]

\begin{claim}
For each $i,k<\omega$, $\pos(\bar{s}^i\conc\bar{s}^{i+1}\conc\ldots\conc 
\bar{s}^{i+k})\subseteq B_{i-1}\subseteq A_{\ell_{i-1}}$.
\end{claim}

\begin{proof}[Proof of the Claim]
  Induction on $k$; fully parallel to \ref{cl7}.
\end{proof}

Now the theorem readily follows. 
\end{proof}

\begin{conclusion}
  Assume that $(K,\Sigma)$ is a tight FFCC pair, $J$ is an R--ideal on
  $\omega$ and $\bar{t}\in\PCJ$. Suppose also that $(K,\Sigma)$ has
  $J$--bigness for $\bar{t}$. For $n<\omega$ let $k_n<\omega$ and let
  $d_n:\post(\bar{t}\upl n)\longrightarrow k_n$. Then there are
  $\bar{s}\in\PCJ$ and $0=k(0)<k(1)<k(2)<\ldots <\omega$ and $\langle
  c_i:i<\omega\rangle$ such that
  \begin{itemize}
\item $\bar{t}\leq \bar{s}$, $m^{s_0}_{\dn}=m^{t_0}_\dn$, and 
\item for each $i,n<\omega$,\\
if $s_{k(i)}\in\Sigt(\bar{t}\upl n)$, $i<j<\omega$ and
  $f\in\pos(s_{k(i)},s_{k(i)+1},\ldots,s_{k(j)-1})$,\\
then $d_n(f)=c_i$. 
  \end{itemize}
\end{conclusion}

\begin{corollary}
Let $\bH^*:\omega\longrightarrow\omega\setminus\{0\}$ be increasing,
$\bbZ^*=\bigcup\limits_{n<\omega}\prod\limits_{i<n}\bH^*(i)$ and let $J$ be
an R--ideal on $\omega$. Suppose that $\bbZ^*=C_0\cup\ldots\cup C_L$,
$L<\omega$. Then there are sequences $\langle k_i,n_i:i<\omega\rangle$ and
$\langle E_i:i<\omega\rangle$ and $\ell\leq L$ such that
\begin{enumerate}
\item[(a)] $0=n_0\leq k_0<n_1\leq\ldots <n_i\leq k_i\leq n_{i+1}\leq\ldots
  <\omega$, $\{k_i:i<\omega\}\in J^+$,\\
and for each $i<\omega$: 
\item[(b)] $\emptyset\neq E_i\subseteq\bH^*(i)$, $|E_{k_i}|=i+1$, and
\[\prod\limits_{j<n_i} E_j\subseteq C_\ell.\]
\end{enumerate}
\end{corollary}


\end{document}